%% file: Main.tex
\documentclass[11pt, reqno]{amsart}

\usepackage[utf8]{inputenc}
\usepackage[T1]{fontenc}
\usepackage[english]{babel}
\usepackage{amsmath, amsthm}
\usepackage{ae}
\usepackage{icomma}
\usepackage{units}
\usepackage{color}
\usepackage{graphicx}
\usepackage{bbm}
\usepackage{caption}
\usepackage{array}
\usepackage[hmarginratio=1:1]{geometry}
\usepackage[hyphens]{url}
\usepackage[pdfpagelabels=false, hidelinks]{hyperref}
\usepackage{mathrsfs}
\usepackage{amssymb, amsfonts}
\usepackage{placeins} 
\usepackage{tikz}
\usetikzlibrary{patterns}
\usepackage{float} 
\usepackage[nodayofweek]{datetime}
\usepackage{enumitem}
\usepackage{mathtools}
\usepackage[numbers, sort]{natbib}
\usepackage{dsfont}
\usepackage{comment}

\mathtoolsset{showonlyrefs=true}

\newcommand{\R}{\ensuremath{\mathbb{R}}}

\renewcommand{\S}{\ensuremath{\mathbb{S}}}
\newcommand{\Haus}{\ensuremath{\mathcal{H}}}
\newcommand{\K}{\ensuremath{\mathcal{K}}}

\newcommand{\eps}{\ensuremath{\varepsilon}}

\newcommand{\LT}[2]{L_{#1, #2}}

\DeclareMathOperator{\dist}{\textnormal{dist}}
\DeclareMathOperator{\signdist}{\delta}
\DeclareMathOperator{\Tr}{Tr}

\DeclareMathOperator{\supp}{supp}

\DeclareMathOperator{\reg}{reg}

\newcommand{\G}{\ensuremath{\mathcal{G}}}
\newcommand{\U}{\ensuremath{{\mathcal{U}_\eps}}}
\newcommand{\I}{\ensuremath{{\mathcal{I}_\eps}}}

\newcommand{\pps}{\hspace{0.75pt}}

\newtheorem{theorem}{Theorem}[section]
\newtheorem{definition}[theorem]{Definition}
\newtheorem{lemma}[theorem]{Lemma}

\newtheorem{corollary}[theorem]{Corollary}

\numberwithin{theorem}{section}
\theoremstyle{remark}
\newtheorem{remark}[theorem]{Remark}

\newcommand{\limplus}{{\mathchoice{\vcenter{\hbox{$\scriptstyle +$}}}
  {\vcenter{\hbox{$\scriptstyle +$}}}
  {\vcenter{\hbox{$\scriptscriptstyle +$}}}
  {\vcenter{\hbox{$\scriptscriptstyle +$}}}
}}
\newcommand{\limminus}{{\mathchoice{\vcenter{\hbox{$\scriptstyle -$}}}
  {\vcenter{\hbox{$\scriptstyle -$}}}
  {\vcenter{\hbox{$\scriptscriptstyle -$}}}
  {\vcenter{\hbox{$\scriptscriptstyle -$}}}
}}
\newcommand{\limpm}{{\mathchoice{\vcenter{\hbox{$\scriptstyle \pm$}}}
  {\vcenter{\hbox{$\scriptstyle \pm$}}}
  {\vcenter{\hbox{$\scriptscriptstyle \pm$}}}
  {\vcenter{\hbox{$\scriptscriptstyle \pm$}}}
}}

\newcommand{\Emink}{\ensuremath{\vartheta}}
\newcommand{\Egp}{\ensuremath{\mu}}

\begin{document}

\title[Two-term spectral asymptotics]{Two-term spectral asymptotics\\ for the Dirichlet Laplacian in a Lipschitz domain}

\author[R. L. Frank]{Rupert L. Frank}
\address{\textnormal{(R. L. Frank)} Mathematisches Institut, Ludwig-Maximilans Universit\"at M\"unchen, Theresinstr. 39, 80333 M\"unchen, Germany, and Department of Mathematics, California Institute of Technology, Pasadena, CA 91125, USA}
\email{r.frank@lmu.de, rlfrank@caltech.edu}

\author[S. Larson]{Simon Larson}
\address{\textnormal{(S. Larson)} Department of Mathematics, KTH Royal Institute of Technology, SE-100 44 Stockholm, Sweden}
\email{simla@math.kth.se}


\begin{abstract}
We prove a two-term Weyl-type asymptotic formula for sums of eigenvalues of the Dirichlet Laplacian in a bounded open set with Lipschitz boundary. Moreover, in the case of a convex domain we obtain a universal bound which correctly reproduces the first two terms in the asymptotics.
\end{abstract}

\thanks{\copyright\, 2019 by the authors. This paper may be
reproduced, in its entirety, for non-commercial purposes.\\
U.S.~National Science Foundation grant DMS-1363432 (R.L.F.) and Swedish Research Council grant no.~2012-3864 (S.L.) is acknowledged. The authors are grateful to an anonymous referee for helpful remarks.}

\maketitle


\section{Introduction and main result}

In this paper we investigate the asymptotic behavior of the eigenvalues of the Dirichlet Laplacian on domains with rough boundary. Besides being of intrinsic interest, this question is relevant for some problems in shape optimization, as we will explain below in some more detail.

One of the central results in the spectral theory of differential operators is Weyl's law~\cite{Weyl_12}. It states that the eigenvalues 
\begin{equation}
  0 < \lambda_1 \leq \lambda_2 \leq \lambda_3 \leq \dots \, , 
\end{equation}
repeated according to multiplicities, of the Dirichlet Laplacian $-\Delta_\Omega$ in an open set $\Omega\subset\R^d$ of finite measure satisfy
\begin{equation}\label{eq:first_Weyl_term}
\# \{ \lambda_k < \lambda \} = \frac{\omega_d}{(2\pi)^d}|\Omega| \lambda^{d/2} + o(\lambda^{d/2})
 \qquad \mbox{as }\lambda\to \infty \, , 
\end{equation}
where $\omega_d$ denotes the measure of the unit ball in $\R^d$. The fact that this asymptotic expansion holds without any regularity conditions on $\Omega$ was shown in \cite{Rozenblum_72}.

In~\cite{Weyl_13} Weyl conjectured that a refined version of the asymptotic formula~\eqref{eq:first_Weyl_term} holds. Namely, he conjectured that
\begin{equation}\label{eq:asym_ivrii}
\# \{ \lambda_k < \lambda \} = \frac{\omega_d}{(2\pi)^d}|\Omega| \lambda^{d/2} - \frac14 \frac{\omega_{d-1}}{(2\pi)^{d-1}} \Haus^{d-1}(\partial \Omega) \lambda^{(d-1)/2} + o(\lambda^{(d-1)/2})
 \qquad \mbox{as }\lambda\to \infty \, .
\end{equation}
Here $\Haus^{d-1}(\partial\Omega)$ denotes the $(d-1)$-dimensional Hausdorff measure of the boundary. This conjecture was proved by Ivrii in \cite{Ivrii_80} under two additional assumptions. The first assumption is that the measure of all periodic billiards is zero and the second assumption is that the boundary of the set is smooth. It is believed, but only known in special cases~\cite{Vassiliev_84,Vassiliev_86}, that the first assumption is always satisfied. Concerning the second assumption, in a series of papers \cite{Ivrii_00, BronsteinIvrii_03, Ivrii_03} Ivrii and co-workers have tried to lower the required assumptions on the boundary of the set. In particular, in \cite{Ivrii_03} the asymptotics~\eqref{eq:asym_ivrii} are proved under the billiard assumption for $C^1$ domains such that the derivatives of the functions describing the boundary have a modulus of continuity $o(|{\log r}|^{-1})$. Without the billiard assumption it is shown that the left side of \eqref{eq:asym_ivrii} differs from the first term on the right side by $O(\lambda^{(d-1)/2})$. This bound, in the smooth case, is originally due to Seeley \cite{Seeley_78,Seeley_80}.

The goal of this paper is to show that an averaged version of the asymptotics~\eqref{eq:asym_ivrii} is valid for any bounded open set with Lipschitz boundary. In order to state this result precisely, we write $x_\limpm = (|x|\pm x)/2$, so that
$$
\Tr(-\Delta_\Omega-\lambda)_\limminus = \sum_{\lambda_k<\lambda} (\lambda-\lambda_k) \,,
$$
and abbreviate
$$
L_d = \frac{2}{2+d} \frac{\omega_d}{(2\pi)^d} \, .
$$
Our main result is

\begin{theorem}\label{thm:MainTheorem}
Let $\Omega\subset\R^d$, $d\geq 2$, be a bounded open set with Lipschitz regular boundary. Then, as $\lambda\to\infty$, 
  \begin{equation}\label{eq:mainintro}
    \Tr(-\Delta_\Omega-\lambda)_\limminus = L_d |\Omega| \lambda^{1+d/2} - \frac{L_{d-1}}{4} \Haus^{d-1}(\partial\Omega) \lambda^{1+(d-1)/2} + o(\lambda^{1+(d-1)/2}) \, .
  \end{equation}
\end{theorem}

We will discuss momentarily in which sense this theorem improves earlier results and sketch the strategy of its proof. Before doing so, we would like to emphasize that the methods that we develop in order to prove Theorem \ref{thm:MainTheorem} can also be used to prove universal, that is, non-asymptotic bounds. For instance, for convex sets we obtain the following bound.

\begin{theorem}\label{thm:asymptotic_ineq_convex}
  Let $\Omega\subset \R^d$, $d\geq 2$, be a convex bounded open set. Then, for all $\lambda>0$, 
  \begin{align}
    & \Bigl| \Tr(-\Delta_\Omega-\lambda)_\limminus - L_d |\Omega| \lambda^{1+d/2} + \frac{L_{d-1}}{4} \Haus^{d-1}(\partial\Omega) \lambda^{1+(d-1)/2} \Bigr| \\
    & \qquad 
    \leq
    C \Haus^{d-1}(\partial\Omega) \lambda^{1+(d-1)/2} \Bigl( r_{in}(\Omega)\sqrt\lambda \Bigr)^{-1/11}\, ,   
  \end{align}
  where the constant $C$ depends only on the dimension.
\end{theorem}

By integration with respect to $\lambda$, Theorem~\ref{thm:asymptotic_ineq_convex} implies a corresponding inequality for $\Tr(e^{t\Delta_\Omega})$ which is valid uniformly for all $t>0$. This improves an earlier result by van den Berg~\cite{vdBerg_84}, where an additional bound on the curvatures was assumed.

In a similar manner, Theorem~\ref{thm:asymptotic_ineq_convex} implies universal upper and lower bounds for $\Tr(H_\Omega)_\limminus^\gamma$ for all $\gamma \geq 1$. The resulting upper bound can be seen as an improvement of an inequality going back to work of Berezin~\cite{Berezin} and Li--Yau~\cite{LiYau_83}. 
Such improved versions of the Berezin--Li--Yau inequality have been the topic of several recent papers~\cite{Melas_03, Weidl_08, KovarikVugalterWeidl_09, GeisingerWeidl_10, GeisingerLaptevWeidl_11, LarsonPAMS, HarrellStubbe_16}. 
Lower bounds in the same spirit are contained in~\cite{HarrellProvenzanoStubbe_18}. In contrast to our Theorem~\ref{thm:asymptotic_ineq_convex}, however, none of these previous upper and lower bounds reproduces correctly the second term in the asymptotics.

A challenging open question from shape optimization theory, which, in part, motivated this work, is whether for fixed $\gamma\geq 0$, a family $(\Omega_{\lambda,\gamma})_{\lambda>0}$ of optimizers of the problem
  \begin{equation}
    \sup\{\Tr(-\Delta_\Omega-\lambda)_\limminus^\gamma: \Omega\subset\R^d \ \text{open}, \ |\Omega|=1\}\,
  \end{equation}
converges as $\lambda\to\infty$ to a ball of unit measure. We refer to \cite{Larson_JST} for more on this problem. The intuition for why the convergence to a ball might be true is that, while the leading term in the asymptotics of $\Tr(-\Delta_\Omega-\lambda)_\limminus^\gamma$ as $\lambda\to\infty$ is fixed due to the constraint $|\Omega|=1$, maximizing the second term leads to minimizing $\Haus^{d-1}(\partial\Omega)$ under the constraint $|\Omega|=1$. By the isoperimetric inequality the unique solution to this problem is a ball of unit measure. The difficulty with making this intuition rigorous is that one needs the asymptotics of $\Tr(-\Delta_\Omega-\lambda)_\limminus^\gamma$ not only for a fixed domain $\Omega$, but rather for a family of domains $\Omega_{\lambda,\gamma}$ depending on $\lambda$ with a priori no information concerning their geometry.

While we have not been able to answer this question in full generality, we did prove the corresponding result for a similar optimization problem with an additional convexity constraint and $\gamma\geq 1$. Namely, as a corollary of Theorem \ref{thm:asymptotic_ineq_convex} we obtain

\begin{corollary}\label{shapeopt}
  Let $\gamma \geq 1$. For $\lambda>0$ let\/ $\Omega_{\lambda, \gamma}$ denote any extremal domain of the shape optimization problem
  \begin{equation}
    \sup\{\Tr(-\Delta_\Omega-\lambda)_\limminus^\gamma: \Omega\subset\R^d \ \text{convex open}, \ |\Omega|=1\}\, .
  \end{equation}
  Then, up to translation, $\Omega_{\lambda, \gamma}$ converges in the Hausdorff metric to a ball of unit measure as $\lambda \to \infty$.
\end{corollary} 

\begin{proof}
Let $\mathcal K$ be the set of all non-empty, bounded convex open sets in $\R^d$. This is a metric space with respect to the Hausdorff metric. In order to prove the corollary, by~\cite[Proposition~4.1]{Larson_JST} we only need to show that the asymptotic expansion
\begin{equation}\label{eq:two_term_weyl_lambda}
  \Tr(-\Delta_\Omega-\lambda)_\limminus^\gamma = \LT{\gamma}{d}|\Omega|\lambda^{\gamma+d/2}- \frac{1}{4}\LT{\gamma}{d-1}\Haus^{d-1}(\partial\Omega)\lambda^{\gamma+(d-1)/2}+o(\lambda^{\gamma+(d-1)/2})\, , 
\end{equation}
as $\lambda \to \infty$, holds uniformly on compact subsets of $\K$. 
Here
$$
\LT{\gamma}{d}= \frac{\Gamma(\gamma+1)}{(4\pi)^{d/2}\Gamma(\gamma+1+d/2)} \, .
$$

Recall the Aizenman--Lieb identity~\cite{AizenmanLieb_78}: for $0\leq \gamma_1 <\gamma_2$ and $\lambda \geq 0,$
\begin{equation}\label{eq:Aizenman_Lieb}
  \Tr(-\Delta_\Omega-\lambda)_\limminus^{\gamma_2} = B(1+\gamma_1, \gamma_2-\gamma_1)^{-1}\int_0^\lambda \tau^{\gamma_2-\gamma_1-1}\Tr(-\Delta_\Omega-(\lambda-\tau))_\limminus^{\gamma_1}\,d\tau\, ,
\end{equation}
where $B$ denotes the Euler Beta function.

By~\eqref{eq:Aizenman_Lieb} it suffices to prove the uniform asymptotics~\eqref{eq:two_term_weyl_lambda} for $\gamma=1$. 
Since $|\Omega|$ and $\Haus^{d-1}(\partial\Omega)$ are continuous on $\mathcal K$, they are bounded on compact subsets of $\mathcal K$. Therefore it suffices to prove~\eqref{eq:two_term_weyl_lambda} uniformly for sets $\Omega$ with bounded $|\Omega|$ and $\Haus^{d-1}(\partial\Omega)$. This follows from Theorem~\ref{thm:asymptotic_ineq_convex} together with the fact that one can bound $r_{in}(\Omega)$ from below in terms of $|\Omega|$ and $\Haus^{d-1}(\partial\Omega)$, see \eqref{eq:inradius_bound}.
\end{proof}

\begin{remark}
In fact, the convergence in Corollary \ref{shapeopt} holds not only for maximizers, but also for almost-maximizers $(\Omega_{\lambda, \gamma})_{\lambda>0}$ in the sense that $\Omega_{\lambda, \gamma}\subset\R^d$ is convex, open with $|\Omega_{\lambda, \gamma}|=1$ and 
$$
\limsup_{\lambda\to\infty} \lambda^{-\gamma-(d-1)/2} \bigl( \Tr(-\Delta_{\Omega_{\lambda, \gamma}}-\lambda)_\limminus^\gamma - S_\gamma \bigr) \geq 0 \, , 
$$
where $S_\gamma$ denotes the supremum in the corollary. This follows by a straightforward adaptation of the arguments above and in~\cite[Proposition~4.1]{Larson_JST}. 
\end{remark}

Let us now return to discussing Theorem~\ref{thm:MainTheorem}. This theorem improves earlier results from~\cite{FrankGeisinger_11, FrankGeisinger_12} where the asymptotics were shown for sets with $C^{1, \alpha}$ and $C^1$ boundary, respectively. As we will explain below in more detail, the technique of flattening the boundary from~\cite{FrankGeisinger_11, FrankGeisinger_12} cannot be used in the case of Lipschitz boundary, but a different and more robust technique is needed.

The Lipschitz condition on the boundary is essentially an optimal assumption. On the one hand, the result is optimal in the H\"older scale (because there are sets with $C^{0, \alpha}$ boundary for $\alpha<1$ for which $\Haus^{d-1}(\partial\Omega)$ is infinite) and on the other hand, the asymptotics~\eqref{eq:mainintro} are not valid for arbitrary sets for which $\Haus^{d-1}(\partial\Omega)$ is finite (for instance, for a ball divided in two pieces by a hyperplane the piece of the hyperplane contributes once to the measure of the boundary, but should contribute twice to the asymptotics). 

Moreover, within Lipschitz domains the error term $o(\lambda^{1+(d-1)/2})$ is the best possible on the algebraic scale: for any $\eps>0$ one can construct a Lipschitz domain $\Omega$ such that
\begin{equation}
  \limsup_{\lambda\to \infty}\,\lambda^{-1-(d-1)/2+\eps}\Bigl|\Tr(-\Delta_\Omega-\lambda)_\limminus-L_d|\Omega|\lambda^{1+d/2}+ \frac{L_{d-1}}{4}\Haus^{d-1}(\partial\Omega)\lambda^{1+(d-1)/2}\Bigr|=\infty\, .
\end{equation}
This follows by integration with respect to $\lambda$ from a construction mentioned in~\cite{Brown_93}.

Two-term spectral asymptotics under a Lipschitz assumption go back to the work \cite{Brown_93} by Brown, where it is shown that
\begin{equation}\label{eq:heat_asymptotics}
\Tr e^{t\Delta_\Omega} = \sum_{k\geq 1} e^{-t\lambda_k} = (4\pi t)^{-d/2}\Bigl(|\Omega|- \frac{\sqrt{\pi}}{2}\Haus^{d-1}(\partial\Omega)t^{1/2}+o(t^{1/2})\Bigr)
\qquad \mbox{as }t\to 0^\limplus \, .
\end{equation}
Note that \eqref{eq:heat_asymptotics} is an Abel-type average of \eqref{eq:asym_ivrii}, whereas \eqref{eq:mainintro} is a Ces\`aro-type average. It is well-known and easy to see that the asymptotics in \eqref{eq:mainintro} imply those in \eqref{eq:heat_asymptotics}, but not vice versa. The key insight in \cite{Brown_93} was to use ideas from geometric measure theory to decompose a neighborhood of the boundary into a `good' part and a `bad' part with sufficiently precise control on the size of the bad part. Inserting well-known pointwise bounds on the heat kernel into this decomposition one obtains \eqref{eq:heat_asymptotics}. While Brown's decomposition of a neighborhood of the boundary also plays an important role in our proof of \eqref{eq:mainintro}, we are facing the additional difficulty that we cannot work on a pointwise level. Thus, our main task is to show that Brown's geometric measure theory arguments can be combined with the technique of local trace asymptotics used in \cite{FrankGeisinger_11, FrankGeisinger_12}.

Let us sketch the overall strategy of the proof. As in~\cite{FrankGeisinger_11, FrankGeisinger_12} we first localize the operator $-\Delta_\Omega$ into balls whose size varies depending on the distance to $\Omega^c$. (As an aside we point out that our choice of the size of the balls here differs from that in~\cite{FrankGeisinger_11, FrankGeisinger_12}. It is both simpler and has a natural scaling behavior which is crucial for the proof of the uniform inequality in Theorem \ref{thm:asymptotic_ineq_convex}.) There are four different types of balls:
\begin{enumerate}[label=(\roman*)]
  \item\label{case:localizationI} $B\subset \Omega$, i.e.\ we have localized in the bulk of $\Omega$.
  \item\label{case:localizationII} $B\cap \Omega$ is empty, i.e.\ we have localized outside $\Omega$ (here the localized operator is trivially zero).
  \item\label{case:localizationIII} $B\cap \partial\Omega$ is non-empty and is in a certain sense well-behaved.
  \item\label{case:localizationIIII} cases~\ref{case:localizationI}-\ref{case:localizationIII} fail, i.e.\ the set $B\cap \partial\Omega$ is non-empty and fails to be well-behaved in the sense of~\ref{case:localizationIII}.
\end{enumerate}
Balls of type (i) are handled as in~\cite{FrankGeisinger_11, FrankGeisinger_12} and those of type (ii) are trivial. The precise sense in which balls of type (iii) and (iv) are distinguished follows the geometric construction due to Brown~\cite{Brown_93}.

Our analysis diverges from that in~\cite{FrankGeisinger_11, FrankGeisinger_12} when it comes to treating the region near the boundary. In~\cite{FrankGeisinger_11, FrankGeisinger_12} the types (iii) and (iv) were not distinguished. There, the bounds rely on the fact that if the boundary is sufficiently regular, then one can locally make a change of coordinates mapping the boundary to a hyperplane while retaining control of how the Laplacian is perturbed under this mapping. For Lipschitz boundaries this method cannot work; flattening the boundary requires a Lipschitz change of coordinates and can thus result in large perturbations of the Laplacian. 

The idea of distinguishing types (iii) and (iv) is in the spirit of Brown's decomposition of a neighborhood of the boundary into a large `good' and a small `bad' part. Essentially, Brown's geometric construction tells us in a quantitative manner that at a sufficiently small scale, the boundary is in most regions well approximated by a hyperplane. For these approximating hyperplanes we can proceed as in the smooth case. However, we are still left with controlling the error from the hyperplane approximation. This is dealt with by proving precise local spectral asymptotics for circular cones (which are the content of Lemma~\ref{lem:LocalAsympCone}).

This concludes our sketch of the proof of Theorem \ref{thm:MainTheorem}. We would like to emphasize that the methods that we develop in this paper are not limited to the situation at hand. In particular, the following three generalizations seem possible:

(1) For our proof it is not crucial that the boundary around \emph{any} point can be represented as a Lipschitz graph. For instance, we could treat domains with a finite number of cusps and also domains with slits (the second term in the asymptotics~\eqref{eq:mainintro} should be modified so that the measure of a slit is counted twice).

(2) Uniform inequalities similar to that in Theorem \ref{thm:asymptotic_ineq_convex} are probably valid also for other classes of domains. The essential ingredients here are Lemmas \ref{lem:boundary_region_convex} and \ref{lem:explicit_G_convex}. For example, analogues of these lemmas can probably be established for sets satisfying a uniform inner and outer ball condition. For such sets uniform bounds for the heat trace were shown in~\cite{vdBerg_87}.

(3) Ba\~nuelos, Kulczycki and Siudeja \cite{BanuelosEtAl_09} have generalized Brown's results for the heat kernel to the case of the fractional Laplacian. Similarly, \cite{FrankGeisinger_16} generalizes the results from  \cite{FrankGeisinger_11} for eigenvalue sums to the case of the fractional Laplacian. Combining these techniques one can probably extend the results in the present paper to the case of the fractional Laplacian.

\subsection*{Structure of the paper}

We begin by introducing some notation, recalling the machinery developed in~\cite{FrankGeisinger_11, FrankGeisinger_12} and proving some corollaries thereof. This is done in Section~\ref{sec:preliminaries}. In Section~\ref{sec:Geometric_constructions} we adapt the geometric constructions of~\cite{Brown_93} to the problem considered here. Section~\ref{sec:MainProof} is dedicated to the proof of Theorem~\ref{thm:MainTheorem} using the tools developed in Sections~\ref{sec:preliminaries} and~\ref{sec:Geometric_constructions}. We end the paper with the proof of Theorem~\ref{thm:asymptotic_ineq_convex} in Section~\ref{sec:uniformity_convex}.


\section{Notation and Preliminaries}
\label{sec:preliminaries} 

Throughout the paper we let $\dist(\, \cdot \, , \cdot\, )$ denote the distance between two sets in $\R^d$ (possibly singletons), that is, 
\begin{equation}
  \dist(A, B) = \inf_{x\in A,\, y\in B}|x-y|\, .
\end{equation}
Given a Lipschitz set $\Omega$ define $\signdist_\Omega(\, \cdot\, )$, the signed distance function of $\Omega$, by
\begin{equation}
  \signdist_\Omega(x) = \dist(x, \Omega^c)-\dist(x, \Omega)\, .
\end{equation}
Note that $\signdist_\Omega(\, \cdot\, )$ and $\dist(\, \cdot\, , \partial\Omega)$ satisfy almost everywhere
\begin{equation}\label{eq:Eikonal}
  |\nabla\!\signdist_\Omega(x)|=1\, , \qquad |\nabla\! \dist(x, \partial\Omega)|=1\, .
\end{equation}
Define also the inradius of $\Omega\subset \R^d$ by 
\begin{equation}
  r_{in}(\Omega)= \sup_{x\in \Omega} \dist(x, \Omega^c)\, .
\end{equation}

We recall that for a Lipschitz domain $\Omega\subset \R^d$ the functions defined by
\begin{align}
  \Emink_{inner}(\Omega, t)&= \frac{|\{u\in \Omega: \dist(u, \partial\Omega)<t\}|}{t\Haus^{d-1}(\partial\Omega)}-1\, , \\
  \Emink_{outer}(\Omega, t)&= \frac{|\{u\in \Omega^c: \dist(u, \partial\Omega)<t\}|}{t\Haus^{d-1}(\partial\Omega)}-1
\end{align}
are both $o(1)$ as $t\to 0^\limplus$~\cite{AmbrosioEtAl_08}.
In what follows we shall suppress~$\Omega$ in the notation and let this dependence be understood implicitly. We also define 
\begin{equation}\label{eq:defemink}
  \overline\Emink(t)=\frac{1}{2}\sup_{t_1, t_2\leq t}\bigl(|\Emink_{inner}(t_1)|+|\Emink_{outer}(t_2)|\bigr)
\end{equation}
so that
\begin{equation}\label{eq:Minkowski_content}
  \biggl|\frac{|\{u\in \R^d: \dist(u, \partial\Omega)<t\}|}{2t \Haus^{d-1}(\partial\Omega)}-1\biggr|\leq \overline\Emink(t)\, .
\end{equation}
The main contributions to the error term of Theorem~\ref{thm:MainTheorem} can be understood in terms of~$\Emink_{inner}(t), \Emink_{outer}(t)$ and $\overline\Emink(t)$.

In the following it will be convenient to introduce the operator 
\begin{equation}
  H_\Omega  =  -h^2 \Delta_\Omega - 1 \qquad\text{in}\ L^2(\Omega)
\end{equation}
with Dirichlet boundary conditions, depending on a parameter $h>0$. Technically,  
$H_\Omega$ is defined as a self-adjoint operator in $L^2(\Omega)$ via the quadratic form $\int_\Omega (h^2 |\nabla u|^2-|u|^2)\, dx$ with form domain $H^1_0(\Omega)$. We have
$$
\Tr(H_\Omega)_\limminus = h^2 \sum_{\lambda_k<h^{-2}} (h^{-2} - \lambda_k) = h^2 \Tr(-\Delta_\Omega-h^{-2})_\limminus \, , 
$$
and therefore the asymptotics in Theorem \ref{thm:MainTheorem} as $\lambda\to\infty$ can be rephrased equivalently as asymptotics for $\Tr(H_\Omega)_\limminus$ as $h \to 0^\limplus$. Similarly, the universal bound in Theorem \ref{thm:asymptotic_ineq_convex} can be rephrased equivalently as a universal bound for $\Tr(H_\Omega)_\limminus$.

For $\phi\in C^\infty(\R^d)$, define $\phi H_\Omega\phi$ as a self-adjoint operator in $L^2(\Omega)$ via the quadratic form $\int_\Omega (h^2 |\nabla (\phi u)|^2-|\phi u|^2)\, dx$ with form domain $H^1_0(\Omega)$.

Let us recall three results from \cite{FrankGeisinger_11, FrankGeisinger_12} concerning localized traces of $H_\Omega$.

\begin{lemma}[{Localized Berezin--Li--Yau inequality~\cite[Lemma~2.1]{FrankGeisinger_11}}]
  \label{lem:BerezinLiebLiYau}
  Let $\phi \in C_0^\infty(\R^d)$. Then, for all $h>0$, 
  \begin{equation}
    \Tr(\phi H_\Omega \phi)_\limminus \leq L_d h^{-d}\int_\Omega \phi^2(x)\, dx\, .
  \end{equation}
\end{lemma}

\begin{lemma}[{\cite[Proposition~1.2]{FrankGeisinger_11}}]\label{lem:AsympBulk}
  Let $\phi\in C_0^\infty(\Omega)$ have support in a ball of radius $l>0$ and satisfy
  \begin{equation}
    \|\nabla \phi\|_{L^\infty}\leq Ml^{-1}\, .
  \end{equation} 
  Then, for all $h>0$, 
  \begin{equation}
    \Bigl|\Tr(\phi H_\Omega \phi)_\limminus - L_d h^{-d}\int_\Omega \phi^2(x)\, dx\Bigr| \leq C l^{d-2}h^{-d+2}\, , 
  \end{equation}
  with a constant $C$ depending only on $M$ and $d$.
\end{lemma}

\begin{lemma}[{\cite[Proposition~1.3]{FrankGeisinger_11}, \cite[Proposition~2.3]{FrankGeisinger_12}}]\label{lem:LocalAsympGraph}
  Let $\phi\in C_0^\infty(\R^d)$ have support in a ball of radius $l>0$ and satisfy
  \begin{equation}
    \|\nabla \phi\|_{L^\infty}\leq Ml^{-1}\, .
  \end{equation}
  Assume that $\partial\Omega \cap \supp \phi$ can be represented as a graph $x_d=f(x')$ and that there is a point $(y', y_d)\in\partial\Omega\cap\supp\phi$ with $\nabla f(y')=0$ and 
  $$
  |\nabla f(x')|\leq \omega(|x'-y'|)
  \qquad\text{for all}\ (x', x_d)\in\partial\Omega\cap\supp\phi \, , 
  $$
  where $\omega\colon [0, \infty) \to [0, \infty)$ is non-decreasing and $\lim_{\delta\to0^\limplus}\omega(\delta)=0$. Then, if $\omega(l)\leq C_d$ and $0<h\leq l$, 
  \begin{equation}
    \biggl| \Tr(\phi H_\Omega \phi) - L_d h^{-d}\int_\Omega \phi^2(x)\, dx + \frac{L_{d-1}}{4}h^{-d+1}\int_{\partial\Omega}\phi^2(x)\, d\Haus^{d-1}(x)\biggr| \leq 
    C \frac{l^d}{h^d}\biggl( \frac{h^2}{l^2}+\omega(l)\biggr)\, , 
  \end{equation}
  where the constant $C_d$ is universal and the constant $C$ depends only on $M$ and $d$.
\end{lemma}

\begin{remark}\label{rem:fg}
This result appears in \cite{FrankGeisinger_11} in the special case $\omega(\delta)=C\delta^\alpha$. The case of a general function $\omega$ appears in \cite{FrankGeisinger_12}, but for the Laplacian with Robin boundary conditions. The proof there, however, extends immediately to the case of Dirichlet boundary conditions. Moreover, a slightly stronger assumption on the parametrization is made in these papers, but only the above one is used, see \cite[Equation (4.1)]{FrankGeisinger_12}. Also, the analysis in \cite{FrankGeisinger_11, FrankGeisinger_12} leads to an additional error term $\omega(l)^2 h/l$ in the parentheses on the right side, but since
$$
\frac{\omega(l)^2h}{l} \leq \frac12 \frac{h^2}{l^2} + \frac12 \omega(l)^4 \leq \frac12 \frac{h^2}{l^2} + \frac{C_d^3}2 \omega(l)
$$
this term is controlled by the other two terms in the parentheses. Finally, there are the following two minor changes. In~\cite{FrankGeisinger_11, FrankGeisinger_12} it is stated that the constant $C$ depends, in addition, on $\|\phi\|_{L^\infty}$ and $\Omega$. However, since $\phi$ has support in a ball of radius $l$ one easily finds $|\phi(x)|\leq l \|\nabla\phi\|_{L^\infty}$, so $\|\phi\|_{L^\infty}\leq M$, and an upper bound on $\|\phi\|_{L^\infty}$ was all that entered in the proof in \cite{FrankGeisinger_12}. Moreover, an inspection of the proof shows that the dependence on~$\Omega$ enters only through the modulus of continuity $\omega$ and that, in fact, only $\omega(l)\leq C_d$ is needed.
\end{remark}

Next, we recall a result of Solovej and Spitzer which provides a family of localization functions adapted to a given local length scale.
\begin{lemma}[{\cite[Theorem~22]{SolovejSpitzer}}]\label{solovejspitzer}
Let $\phi\in C_0^\infty(\R^d)$ with support in $\overline{B_1(0)}$ and $\|\phi\|_{L^2}=1$ and let $l$ be a bounded, positive Lipschitz function on $\R^d$ with Lipschitz constant $\|\nabla l\|_{L^\infty}<1$. Let
\begin{equation}
  \phi_u(x)= \phi\Bigl(\frac{x-u}{l(u)}\Bigr) \sqrt{ 1+ \nabla l(u)\cdot \frac{x-u}{l(u)} } \, .
\end{equation}
Then
\begin{equation}
\label{eq:phi_properties1}
\int_{\R^d} \phi_u(x)^2 l(u)^{-d}\, du = 1
\qquad\text{for all}\ x\in\R^d
\end{equation}
and
\begin{equation}\label{eq:phi_properties}
  \|\phi_u\|_{L^\infty}\leq \sqrt 2\, \|\phi\|_{L^\infty} \quad\text{and}\quad \|\nabla \phi_u\|_{L^\infty} \leq Cl(u)^{-1} \|\nabla\phi\|_{L^\infty}
\quad\text{for all}\ u\in\R^d  \, , 
\end{equation}
where the constant $C$ depends only on $(1-\|\nabla l\|_{L^\infty})^{-1}$.
\end{lemma}

\begin{remark}
  Strictly speaking, the functions $\phi_u$ are defined only for almost every $u\in\R^d$, namely, for those where $\nabla l(u)$ exists. Note that if $(x-u)/l(u)\in\supp\phi$, then $|\nabla l(u)\cdot (x-u)/l(u)|\leq \|\nabla l\|_{L^\infty}< 1$. Therefore the square root in the definition of $\phi_u$ is well-defined and $\phi_u\in C_0^\infty(\R^d)$.
\end{remark}

\begin{remark}
  The assumptions of Lemma~\ref{solovejspitzer} are weaker than those in~\cite{SolovejSpitzer}. However, the proof in~\cite{SolovejSpitzer} applies with almost no change, but for completeness we include it below. 
  Moreover, the definition of $\phi_u$ in \cite{SolovejSpitzer} reads
  $$
  \phi_u(x) = l(u)^{d/2} \phi((x-u)/l(u)) \sqrt{J(x, u)} \, , 
  $$
  where $J(x, u)$ is the absolute value of the Jacobi determinant of the map $u\mapsto (x-u)/l(u)$, that is, 
  $$
  J(x, u) = l(u)^{-d} \biggl| {\det\biggl(1 + \nabla l(u)\otimes \frac{x-u}{l(u)} \biggr)} \biggr|\, .
  $$
  Computing the determinant one arrives at the above formula (which will be important for us later on). 
\end{remark}

\begin{proof}[Proof of Lemma~\ref{solovejspitzer}]
  Without loss of generality we assume that $x=0$. In order to prove~\eqref{eq:phi_properties1} we shall show that the map $F\colon \R^d\to \R^d$ given by $F(u)=-u/l(u)$ is a bijection of $F^{-1}(B_1(0))$ onto $B_1(0)$. After this is established the desired equality follows by a change of variables since
  \begin{equation}
    l(u)^{-d}\Bigl(1+ \nabla l(u)\cdot \frac{x-u}{l(u)} \Bigr) = J(x, u),
  \end{equation}
  where $J(x, u)$ is the absolute value of the Jacobi determinant of the map $u\mapsto (x-u)/l(u)$.

  Fix $u\in \R^d$, since $|F(u)|\geq |u|/\|l\|_{L^\infty}$ and $F(0)=0$ there exists a $t\in [-\|l\|_{L^\infty}, 0]$ such that $F(tu)=u$. Consequently $F$ is surjective.

  That the map is injective on $F^{-1}(B_1(0))$ can be seen as follows. Fix $u\neq 0$. We can write $F(tu)=-g(t)u$ where $g\colon \R\to \R$ is a continuous function, indeed $g(t)=t/l(tu)$. Moreover, we claim that $g$ is monotone increasing for all $t$ such that $|F(t u)|=|g(t)||u|< \|\nabla l\|_{L^\infty}^{-1}$, and in particular for $t$ such that $|F(t u)|=|g(t)||u|\leq 1$. For almost every $t$ it holds that
  \begin{align}
    g'(t) =
     l(t u)^{-1}[1-tl(t u)^{-1}u\cdot \nabla l(t u)]
     \geq 
      l(t u)^{-1}[1-|g(t)||u|\|\nabla l\|_{L^\infty}]
      >0,
  \end{align}
  which proves the claim.
  We conclude that $F$ is a bijection from $F^{-1}(B_1(0))$ to $B_1(0)$.

  Differentiating the formula for $\phi_u$ and using $\|\phi\|_{L^\infty}\leq\|\nabla\phi\|_{L^\infty}$ (see Remark~\ref{rem:fg}) one immediately obtains~\eqref{eq:phi_properties}.
\end{proof}

\begin{lemma}[{Localization}]\label{lem:localization}
Let $\phi$ and $l$ be as in Lemma~\ref{solovejspitzer}. Then, for any $\varphi\in C^\infty(\R^d)$ and all $0<h \leq M \min_{\dist(u, \Omega\,\cap\, \supp \varphi)\leq l(u)}l(u)$, 
  \begin{equation}\label{eq:localization_lemma}
  \begin{aligned}
    \Bigl|\Tr(\varphi &H_\Omega \varphi)_\limminus-\int_{\R^d}\Tr(\phi_u\varphi H_\Omega \varphi\phi_u)_\limminus l(u)^{-d}\, du\Bigr| \\
    &\leq 
    C\|\varphi\|_{L^\infty(\Omega)}^2 h^{-d+2}\int_{\dist(u, \Omega\, \cap \, \supp \varphi)\leq l(u)} l(u)^{-2}\, du\, , 
  \end{aligned}
  \end{equation}
  where the constant depends only on $\|\nabla\phi\|_{L^\infty}$, $(1-\|\nabla l\|_{L^\infty})^{-1}, M$ and $d$.
\end{lemma}

For $\varphi\equiv 1$ this is essentially~\cite[Proposition~1.1]{FrankGeisinger_11}. Here we shall need the slightly more general statement above. However, the proof, which is given in Appendix~\ref{appendixA}, is almost identical to that in~\cite{FrankGeisinger_11}.

\begin{remark}
In~\cite{FrankGeisinger_11} the inequality corresponding to~\eqref{eq:localization_lemma} is stated for all $h>0$, however, the proof requires additionally an upper bound on $h/l(u)$. This does not affect the results in~\cite{FrankGeisinger_11} because for an asymptotic result it suffices to apply the statement where this additional assumption is met. 
Nonetheless, in~\cite{FrankGeisinger_11} the inequality is stated for a particular choice of~$l$ for which it can be extended to all $h>0$, if one assumes that a parameter $l_0$ in their construction satisfies $\liminf_{h\to 0^\limplus}l_0/h >0$. This will be proved in Appendix~\ref{appendixA}. 
\end{remark}


With these preparations at hand, we now show how the method of \cite{FrankGeisinger_11} can be used to compute a two-term asymptotic formula for circular cones and their complements.

\begin{lemma}[Precise local asymptotics in cones]
  \label{lem:LocalAsympCone}
  Let $\varphi\in C_0^\infty(\R^d)$ have support in a ball of radius $l>0$ and satisfy
  \begin{equation}\label{eq:grad_bound_lemma3}
    \|\varphi\|_{L^\infty}\leq M\, .
  \end{equation}
Let $0\leq \eps \leq 1/2$ and
  \begin{equation}
    \Lambda_\eps = \{x\in \R^d: x_d< \eps|x|\}\, .
  \end{equation}
  Then, for all\/ $h>0$, 
  \begin{equation}
    \Bigl|\Tr(\varphi H_{\Lambda_\eps}\varphi)_\limminus - L_d  h^{-d}\int_{\Lambda_\eps} \varphi^2(x)\, dx + \frac{L_{d-1}}{4}h^{-d+1}\int_{\partial \Lambda_\eps} \varphi^2(x)\, d\Haus^{d-1}(x)\Bigr| \leq C l^{d-4/3}h^{-d+4/3}\, , 
  \end{equation}
  and
  \begin{equation}
    \Bigl|\Tr(\varphi H_{\Lambda^c_\eps}\varphi)_\limminus - L_d h^{-d}\int_{\Lambda^c_\eps} \varphi^2(x)\, dx + \frac{L_{d-1}}{4}h^{-d+1}\int_{\partial \Lambda^c_\eps} \varphi^2(x)\, d\Haus^{d-1}(x)\Bigr| \leq C l^{d-4/3}h^{-d+4/3}\, , 
  \end{equation}
  where the constant $C$ depends only on $M$ and $d$ and, in particular, not on $\eps$.
\end{lemma}

The error $(l/h)^{d-4/3}$ is probably not sharp, but good enough for our purposes. After the proof we will explain that for $d=2$, our proof actually yields the error $(l/h)^\gamma$ for any $\gamma>0$.

\begin{proof}[Proof of Lemma~\ref{lem:LocalAsympCone}]
  We only prove the first claim of the lemma, the second one follows analogously.
  The idea is to apply the arguments from~\cite{FrankGeisinger_11, FrankGeisinger_12} to the operator $\varphi H_{\Lambda_\eps}\varphi$ instead of $H_{\Lambda_\eps}$.

  Before we continue with the main part of the proof we show that the claimed inequality holds for $h\geq l$. 

  For all $h>0$, Lemma~\ref{lem:BerezinLiebLiYau} implies that
  \begin{align}
    \Bigl|\Tr(\varphi H_{\Lambda_\eps}&\varphi)_\limminus - L_d  h^{-d}\int_{\Lambda_\eps} \varphi^2(x)\, dx + \frac{L_{d-1}}{4}h^{-d+1}\int_{\partial \Lambda_\eps} \varphi^2(x)\, d\Haus^{d-1}(x)\Bigr|\\
    &\leq 2L_d  h^{-d}\int_{\Lambda_\eps} \varphi^2(x)\, dx + \frac{L_{d-1}}{4}h^{-d+1}\int_{\partial \Lambda_\eps} \varphi^2(x)\, d\Haus^{d-1}(x)\\
    &\leq
     C (l^d h^{-d}+l^{d-1}h^{-d+1})\, .
  \end{align}
  Here we used~\eqref{eq:grad_bound_lemma3}, $|\Lambda_\eps \cap B_l|\leq C l^d$, and $\Haus^{d-1}(\partial \Lambda_\eps \cap B_l)\leq Cl^{d-1}$. The last inequality follows by noting that $\Lambda_\eps^c\cap B_l$ is convex and the monotonicity of the measure of the perimeter of convex sets under inclusion.

  Consequently the inequality claimed in the lemma holds for all $h\geq l$. Through the remainder of the proof we assume that $0<h<l$.

  Since $\Lambda_\eps$ is scale invariant, we may and will assume that $l=1$.

  {\bf Step 1:} We derive a local $C^1$ modulus of continuity for $\partial \Lambda_\eps$. We claim that for any $|u|\geq 4 r$ and $B_r(u)\cap \partial \Lambda_\eps \neq \emptyset$ we can choose a system of coordinates $(x', x_d)\in \R^{d-1}\times \R$ such that $\partial \Lambda_\eps\cap B_r(u)$ can be parametrized as the graph $x_d=f(x')$ of a function $f$ such that for some point in $\partial\Lambda_\eps\cap B_r(u)$ with coordinates $(y', y_d)$ and $\nabla f(y')=0$ one has
    \begin{equation}\label{eq:local_C1modulus}
      |\nabla f(x')|\leq C_{d, \eps}\frac{|x'-y'|}{|u|} \, , 
    \end{equation}
  where $C_{d, \eps}$ is uniformly bounded for $0\leq \eps \leq 1/2$. (In fact, the constant here satisfies $C_{d, \eps} = o_{\eps \to 0^\limplus}(1)$, but this will not be relevant for us. In $d=2$ the boundary of $\Lambda_\eps$ consists of two rays and hence $C_{2, \eps} = 0$.)

  Let us prove \eqref{eq:local_C1modulus}. Pick $x_0\in B_{r}(u)\cap \partial\Lambda_\eps$. Then $B_{r}(u)\cap \partial\Lambda_\eps \subset B_{2r}(x_0)\cap \partial\Lambda_\eps$ and $0 \notin B_{2r}(x_0)$. After rescaling and rotating so that $x_0=(1, 0, \ldots, 0)$ and $\Lambda_\eps \subset \{x\in  \R^d: x_d\leq 0\}$ the above inclusions imply that it is sufficient to consider parametrizing $\partial\Lambda_\eps$ as $x_d=f_0(x')$ in the ball $B_{2/3}(x_0)$. Clearly this is possible and $f_0$ is $C^{1, 1}$-regular and thus, by the choice of coordinates, satisfies the estimate
    \begin{equation}
      |\nabla f_0(x')| \leq C_{d, \eps} |x'-x_0'|\, , \quad x_0' = (1, 0, \ldots, 0)\in \R^{d-1}\, , 
    \end{equation}
    where $C_{d, \eps}$ is uniformly bounded for $0\leq \eps \leq 1/2$ and tends to zero as $\eps\to 0^\limplus$. After scaling and translating one obtains~\eqref{eq:local_C1modulus} since by assumption $|x_0|\geq \tfrac 3 4|u|$.

   {\bf Step 2:} We localize the problem. Fix a function $\phi\in C_0^\infty(\R^d)$ with $\supp\phi =\overline{B_1(0)}$ and $\|\phi\|_{L^2}=1$. With a parameter $l_0\in(0, 1]$ depending on $h$ to be determined, set
    \begin{equation}
      l(u)=\frac{1}{2} \min\bigl\{ 2, \max\{ \dist(u, \Lambda_\eps^c), 2l_0\} \bigr\}\, .
    \end{equation}
  Note that $0<l\leq 1$ and, by \eqref{eq:Eikonal}, $\|\nabla l\|_{L^\infty}\leq 1/2$, so Lemma \ref{solovejspitzer} is applicable. Denote by $\phi_u$ the resulting family of functions from that lemma. Assume also that $h\leq l_0$ so that $h\leq l(u)$ for all $u\in \R^d$.

  By Lemma~\ref{lem:localization}, with $M=1$, and a straightforward estimate of the integral remainder we have that
  \begin{equation}\label{eq:localization2}
    \Bigl|\Tr(\varphi H_{\Lambda_\eps}\varphi)_\limminus - 
    \int_{\R^d}\Tr(\phi_u\varphi H_{\Lambda_\eps}\varphi \phi_u)_\limminus l(u)^{-d}\, du\Bigr| \leq C\|\varphi\|^2_{L^\infty} l_0^{-1}h^{-d+2}\, .
  \end{equation}

  {\bf Step 3:} We split
  \begin{equation}\label{eq:integral_split_cone}
  \begin{aligned}
    \int_{\R^d}\Tr(\phi_u\varphi H_{\Lambda_\eps}\varphi \phi_u)_\limminus l(u)^{-d}\, du 
    &=
    \int_{\Lambda^{(1)}}\Tr(\phi_u\varphi H_{\Lambda_\eps}\varphi \phi_u)_\limminus l(u)^{-d}\, du\\
    &\quad +
    \int_{\Lambda^{(2)}}\Tr(\phi_u\varphi H_{\Lambda_\eps}\varphi \phi_u)_\limminus l(u)^{-d}\, du\, , 
  \end{aligned}
  \end{equation}
  where 
  \begin{align}
    \Lambda^{(1)} &= \{u\in \R^d:\ \emptyset\neq\supp\phi_u\varphi \subset\Lambda_\eps \} \, , \\
    \Lambda^{(2)} &= \{u\in \R^d:\ \supp\phi_u\varphi\cap \partial\Lambda_\eps \neq\emptyset \} \, , 
  \end{align}
  and where we used the fact that $\Tr(\phi_u\varphi H_{\Lambda_\eps}\varphi\phi_u)_\limminus=0$ when $\supp\phi_u\varphi\cap\Lambda_\eps=\emptyset$. Since $\supp \varphi$ is contained in a ball of radius $1$ and $\supp \phi_u$ is contained in a ball of radius $l(u)\leq 1$ the set $\Lambda^{(1)}\cup\Lambda^{(2)}$ is contained in a ball of radius $2$. Moreover, it is easy to see that for all $u\in\Lambda^{(2)}$ one has $l(u)\geq\dist(u, \partial\Lambda_\eps)$ and therefore $\dist(u, \partial\Lambda_\eps)\leq l_0$ and $l(u) = l_0$.
  
  Applying Lemma~\ref{lem:AsympBulk} to the first integral in~\eqref{eq:integral_split_cone} and using~\cite[Equation~8]{FrankGeisinger_11} (see also~\eqref{eq:errorint_bulk} below) yields
  \begin{align}
    \int_{\Lambda^{(1)}}\Tr(\phi_u\varphi H_{\Lambda_\eps}\varphi \phi_u)_\limminus l(u)^{-d}\, du
    &=
    L_d h^{-d} \int_{\Lambda^{(1)}}\int_{\Lambda_\eps}\phi_u^2(x)\varphi^2(x) l(u)^{-d}\, dx\, du \\ \label{eq:cone_bulk}
    &\quad +
    O(h^{-d+2})\int_{\Lambda^{(1)}}l(u)^{-2}\, du\\
    &=
    L_d h^{-d} \int_{\Lambda^{(1)}}\int_{\Lambda_\eps}\phi_u^2(x)\varphi^2(x) l(u)^{-d}\, dx\, du + l_0^{-1}O(h^{-d+2}) \, .
  \end{align}

  With a parameter $\delta>0$ to be specified, we split the second integral of~\eqref{eq:integral_split_cone} further, depending on the distance of $u$ from the vertex of~$\Lambda_\eps$, 
  \begin{equation}\label{eq:cone_boundary_int} 
  \begin{aligned}
    \int_{\Lambda^{(2)}}\Tr(\phi_u\varphi H_{\Lambda_\eps}\varphi \phi_u)_\limminus l(u)^{-d}\, du
    &=
    \int_{\Lambda^{(2)} \setminus B_\delta}\Tr(\phi_u\varphi H_{\Lambda_\eps}\varphi \phi_u)_\limminus l(u)^{-d}\, du\\
    &\quad +
    \int_{\Lambda^{(2)} \cap\pps B_\delta}\Tr(\phi_u\varphi H_{\Lambda_\eps}\varphi \phi_u)_\limminus l(u)^{-d}\, du \, .
  \end{aligned}
  \end{equation}
 By Lemma~\ref{lem:BerezinLiebLiYau} the second integral is small, that is, 
 \begin{equation}\label{eq:BLLY_close_to_vertex}
  \begin{aligned}
    \int_{\Lambda^{(2)} \cap\pps B_\delta}\Tr(\phi_u\varphi H_{\Lambda_\eps}\varphi \phi_u)_\limminus l(u)^{-d}\, du 
    &\leq 
    L_d h^{-d}\int_{\Lambda^{(2)} \cap\pps B_\delta}\int_{\Lambda_\eps}\phi_u^2(x)\varphi^2(x) l(u)^{-d}\, dx\, du\\
    &\leq 
    C h^{-d} |\Lambda^{(2)}\cap B_\delta| \leq C h^{-d}\delta^{d-1}l_0 \, .
  \end{aligned}
  \end{equation}
  In the last inequality we used the fact that $\Lambda^{(2)}$ is contained in an $l_0$-neighborhood of~$\partial \Lambda_\eps$. For later purposes we also record that
  \begin{equation}\label{eq:vertex_integrals}
    \int_{\Lambda^{(2)} \cap B_\delta}\!\biggl(
      \int_{\Lambda_\eps} \phi_u^2(x)\varphi^2(x)\pps dx
      + h \int_{\partial \Lambda_\eps} \phi_u^2(x)\varphi^2(x)\pps d\Haus^{d-1}(x)
    \!\biggr)l(u)^{-d}\pps du 
    \leq
    C \delta^{d-1}(l_0 + h)\, , 
  \end{equation}
  where we used again $|\Lambda^{(2)}\cap B_\delta|\leq C l_0 \delta^{d-1}$.

  To treat the remaining term of~\eqref{eq:cone_boundary_int} we apply Lemma~\ref{lem:LocalAsympGraph}. Let $C_{d, \eps}$ and $C_d$ be the constants from Step~1 and Lemma~\ref{lem:LocalAsympGraph}, respectively, and let $\omega(r)=C_{d, \eps} r/|u|$. Finally, set $A=\max\{C_{d, \eps}/C_d, 4\}$.
    
  We claim that, if $\delta\geq A l_0$, then $\omega(l(u))\leq C_d$ and for all $u\in\Lambda^{(2)}\setminus B_\delta$ one can parametrize $\partial\Lambda_\eps\cap B_{l(u)}(u)$ as the graph of a function $f$ and for a point $(y', y_d)\in \partial\Lambda_\eps\cap B_{l(u)}(u)$ one has $\nabla f(y')=0$ and $|\nabla f(x')|\leq \omega(|x'-y'|)$ for all $x'\in \R^{d-1}$.
  
  Indeed, for any $u\in\Lambda^{(2)}\setminus B_\delta$ one has $|u|\geq\delta\geq A l_0 = A l(u)$. Therefore, since $A\geq 4$, according to Step~1 such a parametrization is possible with the above choice of~$\omega$. In particular, $\omega(l(u)) = C_{d, \eps} l(u)/|u| \leq C_{d, \eps}/A$. Since $A\geq C_{d, \eps}/C_d$, the claimed inequality holds.

Since $l_0 \geq h$, we for all $u\in\Lambda^{(2)}$ have $l(u)=l_0 \geq h$ and therefore Lemma~\ref{lem:LocalAsympGraph} yields
\begin{equation}\label{eq:cone_bdry_far}
  \begin{aligned}
    & \int_{\Lambda^{(2)} \setminus B_\delta}\Tr(\phi_u\varphi H_{\Lambda_\eps}\varphi \phi_u)_\limminus l(u)^{-d}\, du \\
&\quad  =
    L_d h^{-d}\int_{\Lambda^{(2)} \setminus B_\delta}\int_{\Lambda_\eps}\phi_u^2(x)\varphi^2(x) l(u)^{-d}\, dx\, du \\
    & \qquad 
    - \frac{L_{d-1}}{4}h^{-d+1}\int_{\Lambda^{(2)} \setminus B_\delta}\int_{\partial\Lambda_\eps}\phi_u^2(x)\varphi^2(x)l(u)^{-d}\, d\Haus^{d-1}(x)\, du\\
    & \qquad
    + O(h^{-d})\int_{\Lambda^{(2)}\setminus B_\delta}
    \biggl(
    \frac{h^2}{l(u)^{2}} + C_{d, \eps} \frac{l(u)}{|u|}\biggr)du \, .
  \end{aligned}
  \end{equation}

Combining \eqref{eq:localization2}, \eqref{eq:integral_split_cone}, \eqref{eq:cone_bulk}, \eqref{eq:cone_boundary_int}, \eqref{eq:BLLY_close_to_vertex}, \eqref{eq:vertex_integrals}, \eqref{eq:cone_bdry_far} and using \eqref{eq:phi_properties1} we obtain
$$
\Tr(\varphi H_{\Lambda_\eps}\varphi)_\limminus = L_d  h^{-d}\int_{\Lambda_\eps} \varphi^2(x)\, dx - \frac{L_{d-1}}{4}h^{-d+1}\int_{\partial \Lambda_\eps} \varphi^2(x)\, d\Haus^{d-1}(x) + \mathcal R
$$
with
  \begin{align}\label{eq:error_term_cone1}
  |\mathcal R| \leq C h^{-d} \biggl( l_0^{-1}h^{2} + \delta^{d-1}(l_0+h)
    + \int_{\Lambda^{(2)}\setminus B_\delta}
    \biggl(
      \frac{h^2}{l(u)^{2}}+ C_{d, \eps} \frac{l(u)}{|u|}
    \biggr)du \biggr)\, .
  \end{align}

Our final task in the proof is to choose $l_0$ and $\delta$ such that the right side here becomes $\leq C h^{-d+4/3}$. By~\cite[Equation~8]{FrankGeisinger_11}, see also~\eqref{eq:errorint_boundary}, 
  \begin{equation}
    h^2 \int_{\Lambda^{(2)}\setminus B_\delta}l(u)^{-2}\, du \leq C l_0^{-1}h^2\, .
  \end{equation}
  To bound the remaining term of the integral we consider two cases:
  \begin{enumerate}[label=\roman*.]
    \item If $\Lambda^{(2)} \cap B_1=\emptyset$, then
      \begin{equation}
        C_{d, \eps}\int_{\Lambda^{(2)}\setminus B_\delta} \frac{l(u)}{|u|}\, du
        \leq C_{d, \eps} \int_{\Lambda^{(2)}\setminus B_\delta} l(u)\, du
        \leq
          C\pps C_{d, \eps} l_0^2 \, .
      \end{equation}

    \item If $\Lambda^{(2)} \cap B_1\neq\emptyset$, then $\Lambda^{(2)} \subset B_5$ and
      \begin{equation}
        C_{d, \eps} \int_{\Lambda^{(2)}\setminus B_\delta}
         \frac{l(u)}{|u|}\, du
        \leq
          C\pps C_{d, \eps} l_0^2 \int_\delta^5 \tau^{-1} \tau^{d-2}\, d\tau
        \leq
        Cl_0^2\times
        \begin{cases}
          0 & \mbox{if } d=2\, , \\
          C_{3, \eps}(1+h \log(\delta^{-1})) & \mbox{if }d=3\, , \\
          C_{d, \eps} &\mbox{if }d\geq 4\, .
        \end{cases}
      \end{equation}
  \end{enumerate}
  In both cases we used the fact that $\Lambda^{(2)}$ is contained in an $l_0$-neighborhood of~$\partial\Lambda_\eps$.

  In conclusion, the right side of~\eqref{eq:error_term_cone1} is bounded by
  \begin{equation}\label{eq:cone_final_errors}
 C h^{-d} \Bigl(  l_0^{-1}h^2+\delta^{d-1}(l_0+h)+ C_{d, \eps}l_0^2(1+h\log(\delta^{-1}))\Bigr)\, , 
  \end{equation}
  where the $\log$ term appears only in $d=3$. Setting $\delta=A l_0$ and $l_0=h^{2/3}$, we obtain the claimed error bound. Note that $1\geq l_0 \geq h$ for $0<h\leq 1$, as required.
\end{proof}

\begin{remark}
  In the two-dimensional case the above argument can be iterated to obtain Lemma~\ref{lem:LocalAsympCone} with an error term of order $l^{\gamma}h^{-\gamma}$ for any $\gamma>0$. Indeed, if one has Lemma~\ref{lem:LocalAsympCone} with error term~$l^{\gamma_0}h^{-\gamma_0}$ for some $\gamma_0 \in (0, 2]$, then one can replace the application of Lemma~\ref{lem:BerezinLiebLiYau} in~\eqref{eq:BLLY_close_to_vertex} by an application of this asymptotic expansion and one can avoid~\eqref{eq:vertex_integrals}. Therefore~\eqref{eq:cone_final_errors} is replaced by $h^{-2} (l_0^{-1}h^2+\delta h^{2-\gamma_0}l_0^{-1+\gamma_0})$. Choosing again $\delta=Al_0$ but now $l_0 = h^{\gamma_0/(1+\gamma_0)}$ yields a two-term expansion with error of order~$l^{\gamma'}h^{-\gamma'}$ with $\gamma'= \tfrac{\gamma_0}{1+\gamma_0}$. Repeating this procedure the exponent $\gamma$ can be made arbitrarily small. In higher dimensions the corresponding idea does not yield an improvement since the term $l_0^{-1}h^2+C_{d, \eps}l_0^2$ in~\eqref{eq:cone_final_errors} can be made no smaller than~$h^{4/3}$.
\end{remark}


\section{Geometric constructions}
\label{sec:Geometric_constructions}

In this section we adapt the geometric ideas used by Brown in~\cite{Brown_93} (see also~\cite{BanuelosEtAl_09}) to the setting considered here.

\begin{definition}
  Let $0<\eps\leq 1$ and $r>0$. A point $p\subset \partial \Omega$ is called $(\eps, r)$-good if the inner unit normal $\nu(p)$ exists and
  \begin{equation}
    B_r(p)\cap \partial\Omega \subset \{x\in \R^d: |(x-p)\cdot \nu(p)|< \eps |x-p|\}\, .
  \end{equation}
  The set of all $(\eps, r)$-good points of $\partial\Omega$ is denoted by $G_{\eps, r}$.
\end{definition}

In other words, $p$ is $(\eps, r)$-good if locally $\partial \Omega$ is contained in the complement of the two-sided circular cone with vertex $p$, symmetry axis $\nu(p)$, and opening angle $\arcsin(\sqrt{1-\eps^2})=\arccos(\eps)$ measured from the axis of symmetry.

Following~\cite{Brown_93, BanuelosEtAl_09} we define a good subset of points near the boundary. In contrast to the constructions in~\cite{Brown_93, BanuelosEtAl_09} this set will contain points both in $\Omega$ and in its complement $\Omega^c$.
\begin{definition}
  Let
  \begin{equation}
    \Gamma_{\eps, r}(\pps p)=\{x\in \R^d : |(x-p)\cdot \nu(p)|> \sqrt{1-\eps^2}|x-p|\}\cap B_{r/2}(p)
  \end{equation}
  and
  \begin{equation}
    \G_{\eps, r} = \bigcup_{p\in G_{\eps, r}}\, \Gamma_{\eps, r}(\pps p) \, .
  \end{equation}
\end{definition}

We emphasize that $\Gamma_{\eps, r}(p)$ differs from the corresponding set defined in~\cite{Brown_93, BanuelosEtAl_09} in several ways. Here we avoid an additional degree of freedom by taking the union over \emph{all} $(\eps, r)$-good points instead of a subset of them, we consider two-sided cones instead of one-sided, and we also choose to truncate the cone at distance $r/2$ instead of $r$.

The two-sided cones appear since we, in contrast to~\cite{Brown_93, BanuelosEtAl_09}, do not work at a pointwise level but at the local length scale given by $l$. In particular, we have a non-trivial contribution to the trace from localizations centered at points $u \notin \Omega$ (see Lemma~\ref{lem:localization}).

The reason for considering smaller cones is to ensure that if $u\in \G_{\eps, r}$ then $\partial\Omega\cap B_{r'}(u)$ stays close to the hyperplane tangent to $\partial\Omega$ at $p$ as long as $r'\leq r/2$. In particular, we shall make use of the following lemma. 

\begin{lemma}\label{lem:vertex_is_close}
  Let $p\in \partial\Omega$ be $(\eps, r)$-good with $0<\eps\leq 1/2$. 
Then for any $u\in \Gamma_{\eps, r}(p)$, 
$$
|u-p| \leq 2 \dist(u, \partial\Omega) \, .
$$
\end{lemma}

\begin{proof}[Proof of Lemma~\ref{lem:vertex_is_close}]
Let $p'\in \partial\Omega$ satisfy $|u-p'|=\dist(u, \partial\Omega)$. Then, since $p\in\partial\Omega$, 
$$
|u-p'|=\dist(u, \partial\Omega)\leq |u-p|<r/2
$$
and so, in particular, $p'\in B_r(p)$. Let $\Lambda=\{y: |(y-p)\cdot \nu(p)|<\eps |y-p|\}$. Then, since $p$ is $(\eps, r)$-good, $p'\in \partial\Omega \cap B_r(p)$ implies that $p'\in\Lambda$. Let $y\in\overline\Lambda\cap \overline{B_r(p)}$ satisfy $|u-y|=\dist(u, \Lambda\cap B_r(p))$. Then, since $p'\in\Lambda\cap B_r(p)$, $|u-y|\leq |u-p'|$. By the choice of $y$ and the construction of $\Gamma_{\eps, r}(p)$ the points $u, p, y$ form a right-angle triangle with the angle between the sides $u-p$ and $y-p$ larger than $\pi/2-2\arcsin(\eps)$. By elementary trigonometry it follows that, for $\eps\in (0, 1/2]$, 
  \begin{equation}
    |u-p'| \geq |u-y| \geq \sin(\pi/2-2\arcsin(\eps))|u- p|=(1-2\eps^2)|u- p| \geq \frac{1}{2}|u-p|\, .
  \end{equation}
This completes the proof.
\end{proof}

The proof of the following result, which is omitted, is based on Rademacher's theorem on almost everywhere differentiability of Lipschitz functions.

\begin{lemma}[{\cite[Section~4]{Brown_93}}]\label{lem:ExistenceGoodSet}
  For any $\eps>0$, 
  \begin{equation}
    \lim_{r\to 0^\limplus}\Haus^{d-1}(\partial\Omega \setminus G_{\eps, r})=0\, .
  \end{equation}
\end{lemma}

It follows that for any fixed $\eps> 0$ we can find $r>0$ small enough so that $G_{\eps, r}$ is non-empty. Furthermore, defining for $\eps>0$
\begin{equation}\label{eq:defmu}
  \Egp_\Omega(\eps, r) = \frac{\Haus^{d-1}(\partial\Omega\setminus G_{\eps, r})}{\Haus^{d-1}(\partial\Omega)} \, , 
\end{equation}
there is an $r>0$ so that $\Egp(\eps, r)$ is arbitrarily small. We shall often write simply~$\Egp$ and leave the dependence on $\Omega$ implicit. For the next lemma we recall that $\overline\Emink$ was defined in~\eqref{eq:defemink}.

\begin{lemma}[{\cite[Proposition~1.3]{Brown_93}, \cite[Lemma~2.7]{BanuelosEtAl_09}}]\label{lem:GoodSetIsBig}
  Let $\eps \in (0, 1]$ and $r>0$. Then there exists an $s_0=s_0(\eps, r, \Omega)>0$ such that for all $s\leq s_0$, 
  \begin{equation}\label{eq:bound_goodsetisbig}
    |\{u \in \R^d: \dist(u, \partial \Omega)<s\}\setminus \G_{\eps, r}| \leq 
    2s (\Egp(\eps, r) +\overline\Emink(s) + \eps^2)\Haus^{d-1}(\partial\Omega)\, .
  \end{equation}
\end{lemma}

\begin{proof}[Proof of Lemma~\ref{lem:GoodSetIsBig}]
  The proof follows closely those of Lemma~2.7 and Proposition~1.3 in~\cite{BanuelosEtAl_09} and~\cite{Brown_93}, respectively. Write
  \begin{equation}\label{eq:GoodSetIsBig_proof1}
  \begin{aligned}
    |\{u\in \R^d: \dist(u, \partial\Omega)<s\}\setminus \G_{\eps, r}| 
    &= |\{u\in \R^d: \dist(u, \partial\Omega)<s\}|\\
    &\quad-|\{u\in \R^d: \dist(u, \partial\Omega)<s\}\cap \G_{\eps, r}|\, .
  \end{aligned}
  \end{equation}
  The first term on the right side can be controlled using~\eqref{eq:Minkowski_content}.
  To bound the second one, for some $\delta>0$ to be determined later, choose $\nu_1, \ldots, \nu_N\in \S^{d-1}$ and disjoint closed sets $F_1, \ldots, F_N \subset G_{\eps, r}$ such that $\Haus^{d-1}(G_{\eps, r}\setminus \bigcup_{i=1}^N F_i)\leq \delta \Haus^{d-1}(G_{\eps, r})$ and $|\nu(p)-\nu_i|\leq \eps$ for all $p\in F_i$. Mimicking the proofs in~\cite{Brown_93, BanuelosEtAl_09} one finds that $p+\rho \nu_i\in \Gamma_{\eps, r}(p)$ for $p\in F_i$ and $-r/2<\rho<r/2$ and that the map $(p, \rho)\mapsto p+\rho\nu_i$ is injective for $p\in F_i$ and $-r/2<\rho<r/2$.

  If $s_0$ is less than or equal to both $r/2$ and $\min_{i\neq j}\dist(F_i, F_j)/2$ one obtains by the area formula~\cite[Theorem~3.2.3]{Federer_book} that, for $0<s\leq s_0$, 
  \begin{equation}\label{eq:GoodSetIsBig_proof2}
  \begin{aligned}
    |\{u\in \R^d: \dist(u, \partial\Omega)<s\}\cap \G_{\eps, r}| &\geq \sum_{i=1}^N |\{p+\rho\nu_i: p\in F_i, -s<\rho<s\}|\\
    &\geq 
   (1-\eps^2/2)\sum_{i=1}^N \int_{\{p+\rho\nu_i: p\in F_i, -s<\rho<s\}} \frac{dx}{\nu_i\cdot\nu(p)} \\
    & = 2s(1-\eps^2/2)\sum_{i=1}^N \Haus^{d-1}(F_i)\\
    &\geq
    2s(1-\eps^2/2)(1-\delta)\Haus^{d-1}(G_{\eps, r})\\
    &=
     2s (1-\Egp(\eps, r))(1-\eps^2/2)(1-\delta)\Haus^{d-1}(\partial\Omega)\, .
  \end{aligned}
  \end{equation}
  
  Combining~\eqref{eq:GoodSetIsBig_proof1},~\eqref{eq:GoodSetIsBig_proof2} and the definition of~$\overline\Emink$ yields
  \begin{equation}
    |\{u\in \R^d: \dist(u, \partial\Omega)<s\}\setminus \G_{\eps, r}| \leq 2s\Haus^{d-1}(\partial\Omega)(1-(1-\Egp(\eps, r))(1-\eps^2/2)(1-\delta)+\overline\Emink(s))\, .
  \end{equation}
  Choosing $\delta=\eps^2/2$ and recalling that $\Egp(\eps, r)\leq 1$ completes the proof.
\end{proof}


\section{Asymptotics for Lipschitz domains}\label{sec:MainProof}

Our goal in this section is to prove the following

\begin{theorem}\label{thm:MainTheorem_alt}
Let $\Omega\subset\R^d$, $d\geq 2$, be a bounded open set with Lipschitz regular boundary. Then, as $h\to 0^\limplus$, 
  \begin{equation}
    \Tr(H_\Omega)_\limminus = L_d |\Omega| h^{-d}-\frac{L_{d-1}}{4}\Haus^{d-1}(\partial \Omega)h^{-d+1}+o(h^{-d+1})\, .
  \end{equation}
\end{theorem}

Clearly, this is equivalent to Theorem~\ref{thm:MainTheorem}. Our proof of Theorem \ref{thm:MainTheorem_alt} depends on three parameters 
$$
\eps_0\in(0, 4] \, , \qquad
\eps\in(0, 1/2] \, , \qquad
r>0
$$
and we shall show that for each such choice of parameters there is an $h_0(\eps_0, \eps, r, \Omega)>0$ such that for all $0<h\leq h_0(\eps_0, \eps, r, \Omega)$ one has
\begin{equation}\label{eq:asymptotic_ineq}
  h^{d-1}\Bigl|\Tr(H_\Omega)_\limminus - L_d |\Omega|h^{-d}+ \frac{L_{d-1}}{4}\Haus^{d-1}(\partial\Omega)h^{-d+1}\Bigr|
  \leq
  C \Bigl(\eps_0^{1/3}+ \frac{\eps}{\eps_0}+ \frac{\overline\Emink(l_0)}{\eps_0}+ \frac{\mu(\eps, r)}{\eps_0}\Bigr)\, , 
\end{equation}
where $C$ is a constant that depends in an explicit way on $\Omega$. Here $\overline\Emink(s)$ and $\mu(\eps, r)$ are the functions from \eqref{eq:defemink} and \eqref{eq:defmu}. Recalling that $\lim_{t\to 0}\overline\Emink(t)=0$ and $\lim_{r\to 0}\Egp(\eps, r)=0$ for any fixed $\eps>0$ (see Lemma \ref{lem:ExistenceGoodSet}), Theorem~\ref{thm:MainTheorem_alt} follows from~\eqref{eq:asymptotic_ineq} by letting $h, r, \eps$ and $\eps_0$ tend to zero in that order. 

There is nothing special about the assumption that $\eps_0\leq 4$. Any choice of upper bound is sufficient to complete the proof and would only result in a change of the constant $C$ in~\eqref{eq:asymptotic_ineq}. However, for our analysis in Section~\ref{sec:uniformity_convex} allowing $\eps_0\in (0, 4]$ will be convenient.

We now give the details of our construction. We introduce a local length scale
\begin{equation}\label{eq:lengthscale_mainproof}
  l(u)=\frac{1}{2}\max\{\dist(u, \Omega^c), 2 l_0\}
\end{equation}
with a parameter $0<l_0\leq r_{in}(\Omega)/2$ that we will write as
$$
l_0 = h/\eps_0 \qquad \mbox{for } 0<h\leq 2r_{in}(\Omega)\, .
$$
Here $\eps_0\in(0, 4]$ is one of the parameters of our construction. We note in passing that the above definition of $l(u)$ is similar, but simpler than that in~\cite{FrankGeisinger_11, FrankGeisinger_12} and has a natural scaling.

Note that $0<l(u)\leq r_{in}(\Omega)/2$ and that, using \eqref{eq:Eikonal}, $\|\nabla l\|_{L^\infty} \leq 1/2$.

Fix a function $\phi\in C_0^\infty(\R^d)$ with $\supp\phi =\overline{B_1(0)}$ and $\|\phi\|_{L^2}=1$. Later on, it will also be important that $\phi$ is radially symmetric.

Lemma \ref{solovejspitzer} now yields a family of functions $(\phi_u)_{u\in\R^d}$ such that $\supp\phi_u =\overline{B_{l(u)}(u)}$ and~\eqref{eq:phi_properties1} and~\eqref{eq:phi_properties} are satisfied.

In what follows we will use the convention that $C$ denotes a constant which may change from line to line but only depends on the dimension and the choice of $\phi$. In particular, we emphasize that it is independent of $\Omega$. Similarly, when we write $O(\, \cdot\, )$ the implicit constant is independent of~$\Omega$ and all the parameters of the construction.

If $h\leq 2 r_{in}(\Omega)$ then 
$$
  \min_{\dist(u, \Omega)\leq l(u)}l(u)= h/\eps_0 \geq h/4\, .
$$ 
Thus, for $0<h\leq 2r_{in}(\Omega)$ we can apply Lemma~\ref{lem:localization}, with $M=4$ and $\varphi \equiv 1$, and reduce our problem to studying the local contributions to the trace $\Tr(\phi_uH_\Omega\phi_u)_\limminus$. (The fact that the integral on the right side of~\eqref{eq:localization_lemma} is indeed negligible for small $\eps_0$ will be proven below in \eqref{eq:localizationerrormain}.)

We now continue our construction and fix the parameters $\eps\in(0, 1/2]$ and $r>0$ and define the sets $\G_{\eps, r}$ and $G_{\eps, r}$ as in the previous section. According to Lemma \ref{lem:ExistenceGoodSet} we may and will assume in the following that given $\eps$, the parameter $r$ is chosen so small that $G_{\eps, r}$ is non-empty.
 
We divide the set of $u\in \R^d$ where $\Tr(\phi_uH_\Omega\phi_u)_\limminus$ is non-zero into three parts, 
\begin{equation}\label{eq:Omega_split}\begin{aligned}
  \Omega_* &= \{u \in \R^d: \supp \phi_u \subset\Omega\}\, , \\
  \Omega_g &= \{u \in \G_{\eps, r}: \supp \phi_u \cap\partial\Omega \neq \emptyset\}\, , \\
  \Omega_b &= \{u \in \R^d\setminus \G_{\eps, r}: \supp \phi_u \cap \partial\Omega\neq \emptyset\}\, .
\end{aligned}\end{equation}
Clearly these three sets are disjoint and $\Tr(\phi_uH_\Omega \phi_u)_\limminus =0$ for $u\notin \Omega_*\cup\Omega_g\cup\Omega_b$. Splitting the integral of Lemma~\ref{lem:localization} according to this partition we have
\begin{equation}\label{eq:integral_split}
  \begin{aligned}
    \int_{\R^d}\Tr(\phi_uH_\Omega \phi_u)_\limminus l(u)^{-d}\, du 
    &=
    \int_{\Omega_*}\Tr(\phi_uH_\Omega \phi_u)_\limminus l(u)^{-d}\, du\\
    &\quad +
    \int_{\Omega_g}\Tr(\phi_uH_\Omega \phi_u)_\limminus l(u)^{-d}\, du\\
    &\quad +
    \int_{\Omega_b}\Tr(\phi_uH_\Omega \phi_u)_\limminus l(u)^{-d}\, du\, .
  \end{aligned}
\end{equation}

Let us pause for a moment and review the overall strategy of our proof. In $\Omega_*$ the effect of the boundary is not felt and a sufficiently precise asymptotic expansion follows from Lemma~\ref{lem:AsympBulk}. By Lemma~\ref{lem:GoodSetIsBig} the set $\Omega_b$ is small and its contribution to the trace is negligible.
The set which is most difficult to analyse is $\Omega_g$. Here the asymptotics in cones from Lemma \ref{lem:LocalAsympCone} will play an important role.

\subsection{Some auxiliary estimates}

To control the error terms appearing in the proof we need to be able to control $l(u)$ on the sets in~\eqref{eq:Omega_split}.

We begin with the following observation, 
\begin{equation}
\label{eq:sets}
\Omega_g\cup\Omega_b = \{ u\in\R^d:\ \dist(u, \partial\Omega)\leq l_0\} \, .
\end{equation}
Indeed, by definition of $\Omega_g$ and $\Omega_b$ and since $\supp\phi=\overline{B_1(0)}$, the set on the left equals $\{u\in\R^d:\ \dist(u, \partial\Omega)\leq l(u)\}$. Therefore we need to prove that for any $u\in\R^d$, one has $\dist(u, \partial\Omega)\leq l_0$ if and only if one has $\dist(u, \partial\Omega)\leq l(u)$. This is trivial if $\dist(u, \Omega^c)\leq 2l_0$, since then $l(u)=l_0$. On the other hand, if $\dist(u, \Omega^c)>2l_0$, then $l(u)=(1/2)\dist(u, \Omega^c)=(1/2)\dist(u, \partial\Omega)$, and therefore neither of the two inequalities holds. This completes the proof of \eqref{eq:sets}.

The equality \eqref{eq:sets} together with~\eqref{eq:Minkowski_content} implies that
\begin{equation}\label{eq:volume_boundary_region}
  |\Omega_g\cup \Omega_b| \leq 2l_0 \Haus^{d-1}(\partial\Omega)(1 + \overline\Emink(l_0))\, .
\end{equation}
Note that it also follows from \eqref{eq:sets} that
$$
l(u)=l_0
\qquad\text{if}\ u\in\Omega_g\cup\Omega_b \, .
$$
Consequently, for any $\alpha\in \R$, 
\begin{equation}\label{eq:errorint_boundary}
  \int_{\Omega_g\cup\Omega_b}l(u)^{\alpha}\, du =  l_0^{\alpha}|\Omega_g\cup \Omega_b| \leq 2\Haus^{d-1}(\partial\Omega)l_0^{1+\alpha}(1+\overline\Emink(l_0))\, .
\end{equation}

We now use \eqref{eq:sets} to bound integrals which will appear as error terms later on. We claim that
\begin{align}
\label{eq:errorint_bulk}
\int_{\Omega_*}l(u)^{-2}\, du \leq C \Haus^{d-1}(\partial\Omega)\bigl[1+\overline\Emink(r_{in}(\Omega))\bigr] l_0^{-1}\, , 
\end{align}
To prove this, we decompose
\begin{equation}
\int_{\Omega_*}l(u)^{-2}\, du = l_0^{-2} |\{u\in\Omega_*: \signdist_\Omega(u)\leq 2l_0\}| + 4 \int_{\signdist_\Omega(u)> 2l_0} \signdist_\Omega(u)^{-2}\, du \, .
\end{equation}
Using \eqref{eq:Eikonal} and the co-area formula and integrating by parts we find
\begin{align}
\int_{\signdist_\Omega(u)> 2l_0} \signdist_\Omega(u)^{-2}\, du
& = \int_{2l_0}^{r_{in}(\Omega)} \Haus^{d-1}(\{u \in \Omega_*: \signdist_\Omega(u)=t\}) t^{-2} \, dt \\
& = 2 \int_{2l_0}^{r_{in}(\Omega)} |\{u \in \Omega_*: \signdist_\Omega(u)\leq t\}| t^{-3} \, dt \\
& \qquad + |\Omega_*| r_{in}(\Omega)^{-2} - \frac{1}{4}|\{u \in \Omega_*: \signdist_\Omega(u)\leq 2l_0 \}|l_0^{-2} \, , 
\end{align}
and therefore
$$
\int_{\Omega_*}l(u)^{-2}\, du \leq 8 \int_{2l_0}^{r_{in}(\Omega)} |\{u \in \Omega: \signdist_\Omega(u)\leq t\}| t^{-3} \, dt + 4 |\Omega| r_{in}(\Omega)^{-2} \, .
$$
The second term on the right side can be bounded by
$$
4 |\Omega| r_{in}(\Omega)^{-2} \leq 2|\Omega|r_{in}(\Omega)^{-1}l_0^{-1}
\leq 2 \Haus^{d-1}(\partial\Omega)\bigl[1+2\overline\Emink(r_{in}(\Omega))\bigr] l_0^{-1} \, .
$$
In order to bound the first term, we use the definition of $\overline\Emink$ and get
\begin{align}
\int_{2l_0}^{r_{in}(\Omega)} |\{u \in \Omega: \signdist_\Omega(u)\leq t\}| t^{-3} \, dt 
& \leq \Haus^{d-1}(\partial\Omega)\bigl[1+2\overline\Emink(r_{in}(\Omega))\bigr] \int_{2l_0}^{r_{in}(\Omega)} t^{-2} \, dt \\
& \leq \frac{1}{2}\Haus^{d-1}(\partial\Omega)\bigl[1+2\overline\Emink(r_{in}(\Omega))\bigr] l_0^{-1} \, .
\end{align}
This completes the proof of~\eqref{eq:errorint_bulk}.

Next, we discuss the localization error coming from~\eqref{eq:localization_lemma}. We claim that
\begin{equation}
\label{eq:localizationerrormain}
h^{-d+2} \int_{\dist(u, \Omega)\leq l(u)} l(u)^{-2}\, du \leq
C \Haus^{d-1}(\partial\Omega)\bigl[1+\overline\Emink(r_{in}(\Omega))\bigr] \eps_0 h^{-d+1} \, .
\end{equation}
Note that this term is negligible for the asymptotics if $\eps_0\ll 1$.

Indeed, taking into account \eqref{eq:sets} this follows from~\eqref{eq:errorint_boundary},~\eqref{eq:errorint_bulk} and the fact that $l_0=h/\eps_0$.


\subsection{Contribution from the bulk\texorpdfstring{ $\Omega_*$}{}}

For the first term on the right side of~\eqref{eq:integral_split}, Lemma~\ref{lem:AsympBulk} and~\eqref{eq:errorint_bulk} yield
\begin{align}
  \int_{\Omega_*}\Tr(\phi_uH_\Omega \phi_u)_\limminus l(u)^{-d}\, du 
  &= 
  \int_{\Omega_*}\biggl(L_d h^{-d}\int_{\Omega}\phi^2_u(x)\, dx + l(u)^{d-2}O(h^{-d+2})\biggr)l(u)^{-d}\, du\\
  &= 
  L_d h^{-d}\int_{\Omega_*}\int_\Omega \phi^2_u(x)l(u)^{-d}\, dx\, du\\
  &\quad  + \Haus^{d-1}(\partial\Omega)\bigl[1+\overline\Emink(r_{in}(\Omega))\bigr] l_0^{-1}O(h^{-d+2})\\
  &=
  L_d h^{-d}\int_{\Omega_*}\int_\Omega \phi^2_u(x)l(u)^{-d}\, dx\, du \\
  &\quad  + \eps_0 \Haus^{d-1}(\partial\Omega)\bigl[1+\overline\Emink(r_{in}(\Omega))\bigr] O(h^{-d+1})\, . 
\end{align}
This is already the desired bound. Note that the second term on the right side is negligible for the asymptotics if $\eps_0\ll 1$.


\subsection{Contribution from the bad part of the boundary\texorpdfstring{ $\Omega_b$}{}}

For the third term on the right side of~\eqref{eq:integral_split}, Lemmas~\ref{lem:BerezinLiebLiYau} and~\ref{lem:GoodSetIsBig} yield
\begin{equation}\label{eq:Contribution_badregion}
\begin{aligned}
  0\leq \int_{\Omega_b}\Tr(\phi_u H_\Omega \phi_u)_\limminus l(u)^{-d}\, du 
  &\leq
  L_d h^{-d}\int_{\Omega_b}\int_\Omega \phi_u^2(x)l(u)^{-d}\, dx\, du\\
  &\leq
  C h^{-d}|\Omega_b| \\
  & \leq C \Haus^{d-1}(\partial\Omega) h^{-d+1}(\Egp(\eps, r)+\overline\Emink(l_0)+\eps^2)/\eps_0 \, , 
\end{aligned}
\end{equation}
Here we used $\Omega_b \subset \{u \in \R^d: \dist(u, \partial\Omega) \leq l_0 \}\setminus \G_{\eps, r}$ and assumed $l_0\leq s_0$ where $s_0$ is the constant from Lemma~\ref{lem:GoodSetIsBig}. The latter condition holds for $h$ small enough depending on $\eps_0$, $\eps$, $r$ and $\Omega$.

The bound \eqref{eq:Contribution_badregion} will be sufficient for us. Note that the term on the right side is negligible for the asymptotics if $(\Egp(\eps, r)+\overline\Emink(l_0)+\eps^2)/\eps_0\ll 1$.


\subsection{Contribution from the good part of the boundary\texorpdfstring{ $\Omega_g$}{}}
\label{sec:Contribution good part of the boundary}

The term coming from $\Omega_g$ is more troublesome to deal with. It is the only term which contributes to the second term of the asymptotic expansion, and thus we need to understand its behavior in more detail.

Let $u\in \Omega_g$. Then by definition there is a $p(u)\in G_{\eps, r}$ such that $u \in \Gamma_{\eps, r}(p(u))$. We define two conical sets associated with $u$, namely, 
\begin{align}
   \I =\I(u)&= \{x \in \R^d: (x-p(u))\cdot \nu(p(u)) > \eps |x-p(u)|\}\, , \\
   \U =\U(u)&= \{x \in \R^d: -(x-p(u))\cdot \nu(p(u)) \geq \eps |x-p(u)|\}^c\, .
\end{align}
We note the inclusions $\I\cap B_r(p) \subseteq \Omega\cap B_r(p) \subseteq \U\cap B_r(p)$ and $\partial\Omega \cap B_r(p)\subset \U \setminus \I$. If $h$ is small enough so that $l_0\leq r/2$ (note that this condition on $h$ depends only on $\eps_0$ and~$r$), then the fact that $l(u)=l_0$ implies that $B_{l(u)}(u)\subset B_{r}(p)$, and so
\begin{align}
\label{eq:inclusionsiu}
\I\cap B_{l(u)}(u) \subseteq \Omega\cap B_{l(u)}(u) \subseteq \U\cap B_{l(u)}(u) \, .
\end{align}

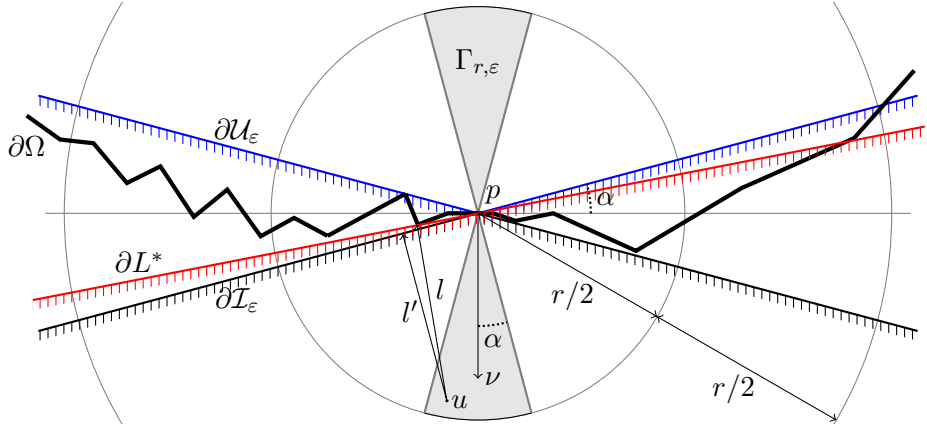
\begin{figure}[ht]
  \centering
  \input{cones}
  \caption{The different sets involved in the construction. Here $\alpha = \arcsin(\eps)$, $p=p(u)$, $\nu=\nu(p(u))$ and $l=\dist(u, \partial\Omega)=\dist(u, \partial L^*)$ and $l'=\dist(u, \partial \I)$. The shaded two-sided truncated cone is the set $\Gamma_{r, \eps}(p)$.}
  \label{fig:cones}
\end{figure}

It is shown in~\cite{BanuelosEtAl_09} that there is a half-space $L^*=L^*(u)$ such that $p(u)\in\partial L^*$, $\dist(u, \partial L^*)$ $= \dist(u, \partial \Omega)$ and $\I\subset L^*(u) \subset \U.$
These inclusions together with \eqref{eq:inclusionsiu} and domain monotonicity imply that
\begin{align}
  \Tr(\phi_u H_\I \phi_u)_\limminus &\leq \Tr(\phi_u H_\Omega \phi_u)_\limminus \leq \Tr(\phi_u H_\U \phi_u)_\limminus \, , \\
  \Tr(\phi_u H_\I \phi_u)_\limminus &\leq \Tr(\phi_u H_{L^*} \phi_u)_\limminus \leq \Tr(\phi_u H_\U \phi_u)_\limminus \, .
\end{align}

Since all the previous arguments hold for any $u\in\Omega_g$ we infer that
\begin{equation}\label{eq:SqueezeCones}
  \begin{aligned}
    \Bigl|\int_{\Omega_g}\Tr(\phi_u H_\Omega \phi_u)_\limminus &l(u)^{-d}\, du - \int_{\Omega_g}\Tr(\phi_u H_{L^*(u)} \phi_u)_\limminus l(u)^{-d}\, du\Bigr|\\ 
    &\leq 
    \int_{\Omega_g}\bigl[\Tr(\phi_u H_{\U(u)} \phi_u)_\limminus-\Tr(\phi_u H_{\I(u)} \phi_u)_\limminus\bigr] l(u)^{-d}\, du \, .
  \end{aligned}
\end{equation}

A technical point here is that the choice of the point $p(u)$ and the half space $L^*(u)$ can be made so that it depends in a measurable way on $u$. The fact that this is possible can be seen by constructing the map $u \mapsto p(u)$ in the following manner. Take a countable dense subset $S$ in $G_{\eps, r}$. The continuity of the map $p \mapsto \Gamma_{\eps, r}(p)$ implies that $\G_{\eps, r}= \cup_{p\in S}\,\Gamma_{\eps, r}(p)$. Choose an ordering of $S$ and define the $u \mapsto p(u)$ by mapping $u$ to the point $p\in S$ which appears first in this ordering. The inverse image of any measurable subset of $\partial\Omega$ is then a countable union of intersections of the sets $\Gamma_{\eps, r}$ which is measurable. The map $u \mapsto L^*(u)$ can be constructed in a similar manner.

We will argue that the second term on the left side of \eqref{eq:SqueezeCones} contains the relevant terms in the asymptotics. In fact, by Lemma~2.3 in~\cite{FrankGeisinger_11} (the case $\omega\equiv 0$ of Lemma~\ref{lem:LocalAsympGraph} above but valid for all $h>0$) it holds that
\begin{equation}
  \Tr(\phi_u H_{L^*}\phi_u)_\limminus = L_d h^{-d}\int_{L^*}\phi_u^2(x)\, dx- \frac{L_{d-1}}{4}h^{-d+1}\int_{\partial L^*}\phi_u^2(x)\, d\Haus^{d-1}(x) + l(u)^{d-2}O(h^{-d+2}) \, .
\end{equation}
Integrating these asymptotics we obtain
\begin{equation}\label{eq:omegagoodhalfspace}
\begin{aligned}
\int_{\Omega_g}\Tr(\phi_u H_{L^*(u)} \phi_u)_\limminus l(u)^{-d}\, du & = L_d h^{-d}\int_{\Omega_g} \int_{L^*(u)}\phi_u^2(x)l(u)^{-d} \, dx\, du \\
& \quad - \frac{L_{d-1}}{4}h^{-d+1} \int_{\Omega_g} \int_{\partial L^*(u)}\phi_u^2(x)l(u)^{-d} \, d\Haus^{d-1}(x)\, du \\
& \quad  + \int_{\Omega_g} l(u)^{-2}\, du\ O(h^{-d+2}) \, .
\end{aligned}
\end{equation}
The first two terms on the right side are almost the terms that we are looking for, namely, 
\begin{equation}
\label{eq:omegagoodhalfspacewanted}
L_d h^{-d}\int_{\Omega_g} \int_{\Omega}\phi_u^2(x)l(u)^{-d} \, dx\, du - \frac{L_{d-1}}{4}h^{-d+1} \Haus^{d-1}(\partial\Omega) \, .
\end{equation}
Note that in the first term on the right side of \eqref{eq:omegagoodhalfspace} we want to replace the domain $L^*(u)$ of the $u$-integration by $\Omega$. Similarly, in the second term we essentially want to replace $\partial L^*(u)$ by $\partial\Omega$ (although eventually we will argue slightly differently). The last term on the right side of \eqref{eq:omegagoodhalfspace} is controlled by~\eqref{eq:errorint_boundary}. 

Thus, in the remainder of this subsection we need to do two things, namely first to control the error between the right side of \eqref{eq:omegagoodhalfspace} and \eqref{eq:omegagoodhalfspacewanted}, and second to bound the term on the right side of \eqref{eq:SqueezeCones}.

\subsubsection{The volume terms}

First we show that the difference between the first term on the right side of \eqref{eq:omegagoodhalfspace} and the first term in \eqref{eq:omegagoodhalfspacewanted} is small. We bound
\begin{align}
  \int_{\Omega_g}\Bigl|\int_\Omega \phi_u^2(x)\, dx - \int_{L^*(u)}\phi_u^2(x)\, dx\Bigr|l(u)^{-d}\, du 
  & \leq
  \int_{\Omega_g}\int_{\Omega \Delta L^*(u)} \phi_u^2(x) l(u)^{-d}\, dx\, du\\ \label{eq:SymmetricDiff}
  & \leq 
  \int_{\Omega_g}\int_{\U(p)\setminus \I(p)} \phi_u^2(x)l(u)^{-d}\, dx\, du\\ 
  &\leq
  C \int_{\Omega_g}|(\U(p)\setminus\I(p))\cap \supp\phi_u|l(u)^{-d}\, du \, .
\end{align}

For $u\in\Omega_g$ we have $l(u)\geq \dist(u, \partial\Omega)$. By Lemma~\ref{lem:vertex_is_close} we find $|u-p(u)|\leq 2l(u)$ and hence
\begin{equation}\label{eq:vertex_is_close}
  (\U(p)\setminus\I(p))\cap B_{l(u)}(u)\subset (\U(p)\setminus \I(p))\cap B_{3l(u)}(p(u)) \, , 
\end{equation}
which in turn implies that
\begin{equation}
  |(\U(p)\setminus\I(p))\cap B_{l(u)}(u)|\leq |(\U(p)\setminus \I(p))\cap B_{3l(u)}(p(u))| \leq C \eps l(u)^d \, .
\end{equation}
Inserting this bound into~\eqref{eq:SymmetricDiff} and recalling \eqref{eq:volume_boundary_region} yields
\begin{align}
  h^{-d}\int_{\Omega_g}|(\U(p)\setminus\I(p))\cap \supp\phi_u|l(u)^{-d}\, du& \leq 
  C h^{-d} \eps |\Omega_g| \\
  &\leq C h^{-d}\eps\pps l_0 \Haus^{d-1}(\partial\Omega)(1+\overline\Emink(l_0))\\
  & = 
  C\eps \eps_0^{-1}h^{-d+1}\Haus^{d-1}(\partial\Omega)(1+\overline\Emink(l_0))\, .
\end{align}
Note that this term is negligible for the asymptotics if $\eps \eps_0^{-1}\ll 1$.


\subsubsection{The boundary terms}

Next, we consider the difference between the second term on the right side of \eqref{eq:omegagoodhalfspace} and the second term in \eqref{eq:omegagoodhalfspacewanted}. We shall show that
\begin{equation}\label{eq:boundary_int_error}
\biggl| \int_{\Omega_g} \int_{\partial L^*(u)}\!\phi_u^2(x)l(u)^{-d} d\Haus^{d-1}(x)\, du - \Haus^{d-1}(\partial\Omega) \biggr|  \leq C \Haus^{d-1}(\partial\Omega) \bigl( \mu(\eps, r) + \overline\Emink(l_0)+\eps^2 \bigr)\, .
\end{equation}
Note that the right side is negligible for the asymptotics if $\mu(\eps, r) + \overline\Emink(l_0)+\eps^2\ll 1$. This is a weaker requirement than the one we met in \eqref{eq:Contribution_badregion}.

Let $u\in\Omega_g$. We know from \eqref{eq:sets} that $l(u)= l_0$ and therefore $\phi_u(x)=\phi((x-u)/l_0)$. 

We define
$$
f(x_d) = \int_{\R^{d-1}} \phi(x', x_d)^2 \, dx' \, .
$$
Let $y\in\partial L^*(u)$ such that $|u-y|=\dist(u, \partial L^*(u))=\dist(u, \partial\Omega)$. Then $\partial L^*(u)=\{ x\in\R^d:\ (x-y)\cdot(u-y)=0\}$ and
\begin{align}
\int_{\partial L^*(u)} \phi_u(x)^2\, d\Haus^{d-1}(x) & = \int_{\partial L^*(u)} \phi\Bigl( \frac{x-y}{l_0} - \frac{u-y}{l_0} \Bigr)^2\, d\Haus^{d-1}(x) \\
& = l_0^{d-1} f(|u-y|/l_0) \, .
\end{align}
The last equality follows by scaling and from the fact that $\phi$ is radial. Since $f$ is even, we can write
$$
f(|u-y|/l_0) = f(\delta_\Omega(u)/l_0) \, .
$$
This proves that
$$
\int_{\Omega_g} \int_{\partial L^*(u)}\phi_u^2(x)l(u)^{-d} \, d\Haus^{d-1}(x)\, du = l_0^{-1} \int_{\Omega_g} f(\delta_\Omega(u)/l_0) \, du \, .
$$

Next, we show that, up to a controllable error, the set $\Omega_g$ on the right side can be replaced by $\R^d$. Indeed, we have
\begin{equation}\label{eq:Adding_bad_boundary}
\begin{aligned}
0& \leq l_0^{-1} \int_{\Omega_b} f(\delta_\Omega(u)/l_0) \, du \leq l_0^{-1} \|f\|_{L^\infty} |\Omega_b| \\
& \leq C \Haus^{d-1}(\partial\Omega)\bigl(\mu(\eps, r)+\overline\Emink(l_0) + \eps^2\bigr) \, , 
\end{aligned}
\end{equation}
where we used the same bound as in~\eqref{eq:Contribution_badregion}. Moreover, since $\phi$ has support in $\overline{B_1(0)}$, $f$ has support in $[-1, 1]$ and therefore \eqref{eq:sets} implies that $f(\delta_\Omega(u)/l_0)=0$ for $u\notin\Omega_g\cup\Omega_b$.

Thus, we are left with analysing
$$
l_0^{-1} \int_{\R^d} f(\signdist_\Omega(u)/l_0) \, du
= l_0^{-1} \int_\R f(t/l_0) \Haus^{d-1}(\{u\in\R^d:\ \signdist_\Omega(u)=t\})\, dt \, .
$$
The identity here comes again from the co-area formula together with \eqref{eq:Eikonal}.

The idea in the following is that $l_0^{-1} f(t/l_0)$ is an approximate delta function at $t=0$. Note that
$$
\int_\R f(x_d)\, dx_d = \|\phi\|_{L^2}^2 = 1 \, .
$$
The following argument is a quantitative, `two-sided' version of a special case of \cite[Proposition 1.1]{Brown_93}. To justify the replacement of $l_0^{-1} f(t/l_0)$ by a delta function write
\begin{align}
& l_0^{-1} \int_0^\infty f(t/l_0) \Haus^{d-1}(\{u\in\R^d:\ \signdist_\Omega(u)=t\})\, dt - (1/2) \Haus^{d-1}(\partial\Omega) \\
& \quad = l_0^{-1} \int_0^\infty f(t/l_0) \frac{d}{dt} \Bigl( |\{ u\in\Omega:\ \signdist_\Omega(u)\leq t\}| - \Haus^{d-1}(\partial\Omega) t \Bigr) dt \\
& \quad = l_0^{-2} \int_0^\infty f'(t/l_0) \Bigl( |\{ u\in\Omega:\ \signdist_\Omega(u)\leq t\}| - \Haus^{d-1}(\partial\Omega) t \Bigr) dt \\
& \quad = l_0^{-2}  \Haus^{d-1}(\partial\Omega) \int_0^\infty f'(t/l_0)\, t\, \Emink_{inner}(t)\, dt \, . 
\end{align}
This, together with a similar formula for $t<0$ and the fact that $f$ is supported in $[-1, 1]$, implies that
\begin{align}
\biggl| &l_0^{-1} \int_\R f(t/l_0) \Haus^{d-1}(\{u\in\R^d:\ \signdist_\Omega(u)=t\})\, dt - \Haus^{d-1}(\partial\Omega) \biggr| \\
& \quad \leq 2 l_0^{-2} \Haus^{d-1}(\partial\Omega) \overline\Emink(l_0) \int_0^\infty |f'(t/l_0)|\, t \, dt \\
& \quad = 2 \Haus^{d-1}(\partial\Omega) \overline\Emink(l_0) \int_0^\infty |f'(x_d)|\, x_d \, dx_d \, .
\end{align}
This completes the proof of \eqref{eq:boundary_int_error}.


\subsubsection{Estimating the error from~\eqref{eq:SqueezeCones}} 

To complete the proof, it remains to control the error made in our local approximation of $B_{l(u)}(u)\cap \Omega$ by $B_{l(u)}(u)\cap L^*(u)$, that is, the right side of~\eqref{eq:SqueezeCones}. We shall show that
\begin{align}
& \int_{\Omega_g}\bigl[\Tr(\phi_u H_{\U(u)} \phi_u)_\limminus-\Tr(\phi_u H_{\I(u)} \phi_u)_\limminus\bigr] l(u)^{-d}\, du \\
& \qquad \leq C \Haus^{d-1}(\partial\Omega)(1+\overline\Emink(l_0))\bigl( \eps \eps_0^{-1} +\eps_0^{1/3} \bigr) h^{-d+1}\, .
\end{align}
Note that in order to show that this term does not interfere with the asymptotics we need to make $\eps\eps_0^{-1}+\eps_0^{1/3}$ small.

Plugging in the asymptotics of Lemma~\ref{lem:LocalAsympCone} we find that
\begin{align}
  \int_{\Omega_g}\bigl[\Tr(\phi_u &H_{\U(u)} \phi_u)_\limminus-\Tr(\phi_u H_{\I(u)} \phi_u)_\limminus\bigr] l(u)^{-d}\, du\\
  &\leq
    L_d h^{-d}\int_{\Omega_g}\int_{\U(p)\setminus\I(p)}\phi^2_u(x)l(u)^{-d}\, dx\, du\\
  &\quad
    - \frac{L_{d-1}}{4}h^{-d+1}\!\int_{\Omega_g}\!\biggl(\int_{\partial \U(p)}\phi^2_u(x)\, d\Haus^{d-1}(x)-\int_{\partial \I(p)}\phi^2_u(x)\, d\Haus^{d-1}(x)\biggr)l(u)^{-d}\, du\\
  &\quad
    + C h^{-d+4/3}\int_{\Omega_g}l(u)^{-4/3}\, du \, .
\end{align}
The first term can be handled as in~\eqref{eq:SymmetricDiff} and is thus $\leq C \Haus^{d-1}(\partial\Omega)(1+\overline\Emink(l_0))\eps \eps_0^{-1} h^{-d+1}$. The third term is $\leq C \Haus^{d-1}(\partial\Omega)(1+\overline\Emink(l_0))\eps_0^{1/3} h^{-d+1}$ by~\eqref{eq:errorint_boundary} and the choice of $l_0$. 

In order to bound the second term, let $H$ denote the hyperplane through $p(u)$ orthogonal to $\nu(p(u))$. Then the map $s\colon \R^{d-1} \to \R$, $x' \mapsto \frac{\eps}{\sqrt{1-\eps^2}}|x'|$, parametrizes $\partial\U$ and $\partial\I$ as graphs over $H$. In coordinates chosen so that $p(u)=0$ and $H=\{(x', 0):\ x'\in \R^{d-1}\}$, we find that
\begin{align}
  \Bigl|\int_{\partial \U}\phi^2_u(x)\, d\Haus^{d-1}(x)-\int_{\partial \I}&\phi^2_u(x)\, d\Haus^{d-1}(x)\Bigr|\\
  &\leq
  \int_{\R^{d-1}} |\phi_u^2(x', s(x'))-\phi_u^2(x', -s(x'))|\sqrt{1+|\nabla s|^2}dx'\\
  &\leq
  \frac{4\eps}{1-\eps^2} \|\phi_u\|_{L^\infty} \|\nabla\phi_u\|_{L^\infty} \int_{B_{3l(u)}} |x'|\, dx' \\
  &\leq
  \frac{C\eps}{\sqrt{1-\eps^2}}l(u)^{d-1} \, , 
\end{align}
where we used $|x'|\leq 3l(u)$ in $\supp \phi_u$, see~\eqref{eq:vertex_is_close}.
Combined with~\eqref{eq:volume_boundary_region} we find that the error coming from the second term of~\eqref{eq:errorint_boundary} is $\leq C \Haus^{d-1}(\partial\Omega)(1+\overline\Emink(l_0)) \eps h^{-d+1}$.

\subsection{Gathering the error terms} 

The proof of Theorem~\ref{thm:MainTheorem_alt} can now be completed by combining the contributions from $\Omega_*, \Omega_b, \Omega_g$ and estimating the localization error from Lemma~\ref{lem:localization}. Note that \eqref{eq:phi_properties1} implies that
$$
\int_{\Omega_*} \int_\Omega \phi_u^2(x)l(u)^{-d}\, dx\, du +
\int_{\Omega_g} \int_\Omega \phi_u^2(x)l(u)^{-d}\, dx\, du +
\int_{\Omega_b} \int_\Omega \phi_u^2(x)l(u)^{-d}\, dx\, du =
|\Omega| \, .
$$
For all $0<h\leq 2 r_{in}(\Omega)$, $r>0$, $\eps\in(0, 1/2]$ and $\eps_0\in (0, 4]$ satisfying
$$
h/\eps_0=l_0\leq \min\bigl\{r/2, s_0, r_{in}(\Omega)/2\bigr\}
$$
(with $s_0=s_0(\eps, r, \Omega)$ given by Lemma~\ref{lem:GoodSetIsBig}) we can conclude that
\begin{equation}\label{eq:mainineq}
\begin{aligned}
&   h^{-d+1}\Bigl|\Tr(H_\Omega)_\limminus - L_d |\Omega|h^{-d}+ \frac{L_{d-1}}{4}\Haus^{d-1}(\partial\Omega)h^{-d+1}\Bigr| \\
& \quad \leq
  C \Haus^{d-1}(\partial\Omega)\biggl[\eps_0 \bigl[1+ \overline\Emink(r_{in}(\Omega))\bigr] + \frac{\mu(\eps, r)+\overline\Emink(l_0)}{\eps_0}
  + \bigl(\eps_0^{-1}\eps + \eps_0^{1/3}\bigr) \bigl[1+\overline\Emink(l_0)\bigr] \biggr]\, , 
\end{aligned}
\end{equation}
where the constant $C$ depends only on the dimension. (Here we have simplified some terms using the fact that $\eps\leq 1/2$ and $\eps_0\leq 4$.) This proves~\eqref{eq:asymptotic_ineq} and therefore concludes the proof of Theorem~\ref{thm:MainTheorem_alt}.
\qed


\section{Uniform asymptotics for convex sets}
\label{sec:uniformity_convex}

Our goal in this section is to prove the following

\begin{theorem}\label{thm:asymptotic_ineq_convex_alt}
  Let $\Omega\subset \R^d$, $d\geq 2$, be a convex bounded open set. Then, for all $h>0$, 
  \begin{equation}
    h^{d-1}\Bigl| \Tr(H_\Omega)_\limminus - L_d |\Omega|h^{-d}+ \frac{L_{d-1}}{4}\Haus^{d-1}(\partial\Omega)h^{-d+1}\Bigr| 
    \leq
    C \Haus^{d-1}(\partial\Omega)\Bigl(\tfrac{h}{r_{in}(\Omega)}\Bigr)^{1/11}\, , 
  \end{equation}
  where the constant $C$ depends only on the dimension.
\end{theorem}

Clearly, this is equivalent to Theorem \ref{thm:asymptotic_ineq_convex}. To prove Theorem~\ref{thm:asymptotic_ineq_convex_alt} we follow the same strategy as in the proof of Theorem~\ref{thm:MainTheorem_alt}. The geometry of $\Omega$ enters into the final inequality~\eqref{eq:mainineq} in that proof via the three quantities $\overline\Emink(l_0)$, $\mu(\eps, r)$ and $s_0(\eps, r, \Omega)$ (the latter as a constraint on the size of $h$).

Our first goal in this section is to show that $\overline \Emink(\Omega, t)$ can be bounded for convex $\Omega$ through~$t/r_{in}(\Omega)$ only. This makes the geometric dependence of the term $\overline\Emink(l_0)$ in \eqref{eq:mainineq} explicit.

It is not so easy to bound $\mu(\eps, r)$ and $s_0(\eps, r, \Omega)$ explicitly, even for convex sets. Our second goal in this section is therefore to prove a replacement of Lemma \ref{lem:GoodSetIsBig} for convex sets where the geometry enters only through $r_{in}(\Omega)$ and $\Haus^{d-1}(\partial\Omega)$.

Having achieved these two goals, a straightforward modification of the proof of Theorem~\ref{thm:MainTheorem_alt} will prove Theorem \ref{thm:asymptotic_ineq_convex_alt}.

Throughout this section we assume that $\Omega\subset\R^d$ is a convex open set. The arguments that follow are based on ideas related to the notion of inner parallel sets. The inner parallel set of $\Omega$ at distance $t$ is defined to be
\begin{equation}\label{eq:innerparallel}
  \Omega_t=\{u\in \Omega: \dist(u, \Omega^c)>t\} \, .
\end{equation}
By~\cite[Theorem~1.2]{Larson_JFA_16} and monotonicity of the measure of the perimeter of convex bodies under inclusions we know that 
\begin{equation}\label{eq:bound_perim_innerparallel}
  \Haus^{d-1}(\partial\Omega)\Bigl(1-\frac{t}{r_{in}(\Omega)}\Bigr)_\limplus^{d-1}\leq \Haus^{d-1}(\partial\Omega_t)\leq \Haus^{d-1}(\partial\Omega)
  \qquad\text{for all}\ t\geq 0 \, .
\end{equation}

Our first application of~\eqref{eq:bound_perim_innerparallel} will be to show that, as claimed above, one has two-sided bounds for $r_{in}(\Omega)$ in terms of $|\Omega|$ and $\Haus^{d-1}(\partial\Omega)$. Indeed, by the co-area formula and~\eqref{eq:Eikonal} one has
\begin{equation}
  |\Omega|= \int_0^{r_{in}(\Omega)}\Haus^{d-1}(\partial\Omega_s)\, ds \, .
\end{equation}
Applying~\eqref{eq:bound_perim_innerparallel} and integrating we find that 
\begin{equation}\label{eq:inradius_bound}
  \frac{|\Omega|}{\Haus^{d-1}(\partial\Omega)}\leq r_{in}(\Omega)\leq  \frac{d|\Omega|}{\Haus^{d-1}(\partial\Omega)} \, .
\end{equation}

\begin{remark}
  It might be worth noting that both bounds in \eqref{eq:inradius_bound} cannot be improved. In the upper bound equality is achieved if $\Omega$ is a ball and, more generally, if and only if $\Omega$ is a form body (see~\cite{Larson_JFA_16, Schneider_14}). In the lower bound equality is asymptotically achieved by $(0, L)^{d-1}\times(0, 1)$ in the limit $L\to \infty$.
\end{remark}

The following lemma achieves the first goal mentioned at the beginning of this section.

\begin{lemma}\label{lem:boundary_region_convex}
Let $\Omega\subset\R^d$ be a convex open set. Then for all $0\leq t\leq r_{in}(\Omega)$, 
  \begin{align}\label{eq:bdry_region_bound_convex}
    |\Emink_{inner}(\Omega, t)|\leq C\frac{t}{r_{in}(\Omega)} \, , 
    \quad |\Emink_{outer}(\Omega, t)|\leq C \frac{t}{r_{in}(\Omega)} \, , 
    \quad \overline \Emink(\Omega, t)\leq C \frac{t}{r_{in}(\Omega)} \, , 
  \end{align}
  where the constants depend only on the dimension.
\end{lemma}

\begin{proof}[Proof of Lemma~\ref{lem:boundary_region_convex}]
  We first bound the measure of $\{u\in \Omega: \dist(u, \Omega^c)<t\}$ from both above and below. Using the co-area formula and~\eqref{eq:Eikonal} in the same manner as above we have that, for $0\leq t\leq r_{in}(\Omega)$, 
  \begin{equation}
    |\{u\in \Omega: \dist(u, \Omega^c)<t\}|=\int_0^t \Haus^{d-1}(\partial\Omega_s)\, ds \, .
  \end{equation}
  By the upper bound in~\eqref{eq:bound_perim_innerparallel} it follows that, for $t\geq 0$, 
  \begin{equation}
    |\{u\in \Omega: \dist(u, \Omega^c)<t\}|\leq t\Haus^{d-1}(\partial\Omega) \, .
  \end{equation}
   Correspondingly, the lower bound in~\eqref{eq:bound_perim_innerparallel} implies that, for $0\leq t\leq r_{in}(\Omega)$, 
  \begin{align}
    |\{u\in \Omega: \dist(u, \Omega^c)<t\}|
    &=
    \int_0^t \Haus^{d-1}(\partial\Omega_s)\, ds \\
    &\geq 
    \Haus^{d-1}(\partial\Omega)\int_0^t\Bigl(1- \frac{s}{r_{in}(\Omega)}\Bigr)^{d-1}\, ds\\
    &=
    \frac{\Haus^{d-1}(\partial\Omega)r_{in}(\Omega)}{d}\Bigl(1-\Bigl(1- \frac{t}{r_{in}(\Omega)}\Bigr)^{d}\Bigr)\\
    &\geq t\Haus^{d-1}(\partial\Omega)\Bigl(1- \frac{d-1}{2r_{in}(\Omega)}t\Bigr) \, .
  \end{align}
  Consequently we find that
  \begin{equation}
    -\frac{d-1}{2r_{in}(\Omega)}\, t \leq \Emink_{inner}(t)\leq 0 \, .
  \end{equation}

  To obtain the corresponding bounds for the measure of $\{x\in \Omega^c: \dist(x, \Omega)<t\}$ we first note that $\{u\in \R^d: \dist(u, \Omega)<t\}$ is convex and its inner parallel set at distance $t$ is $\Omega$. By applying~\eqref{eq:bound_perim_innerparallel} to this set and using $r_{in}(\{u\in \R^d: \dist(u, \Omega)<t\})=r_{in}(\Omega)+t$ we find that
  \begin{align}
    \Haus^{d-1}(\{u\in \R^d: &\dist(u, \Omega)=t\})\Bigl(\frac{r_{in}(\Omega)}{r_{in}(\Omega)+t}\Bigr)^{d-1}\\
    &\leq \Haus^{d-1}(\partial\Omega)\leq \Haus^{d-1}(\{u\in \R^d: \dist(u, \Omega)=t\}) \, .
  \end{align}
Rearranging and arguing as before one finds
  \begin{equation}
    t\Haus^{d-1}(\partial\Omega)\leq |\{u\in \R^d: \dist(u, \Omega)<t\}|\leq  t\Haus^{d-1}(\partial\Omega)\Bigl(1+\frac{2^d-d-1}{d\, r_{in}(\Omega)}t\Bigr) \, , 
  \end{equation}
  and hence
  \begin{equation}
    0\leq \Emink_{outer}(t)\leq \frac{2^d-d-1}{d\, r_{in}(\Omega)}t \, .
  \end{equation}

  By combining the bounds for $\Emink_{inner}$ and $\Emink_{outer}$ one obtains the third inequality in~\eqref{eq:bdry_region_bound_convex}. This completes the proof of the lemma.
\end{proof}

The following lemma achieves the second goal mentioned at the beginning of this section. Note that this is similar to~\eqref{eq:bound_goodsetisbig} but without involving $\mu(\eps, r)$ or $\overline\Emink$ and with an explicit value for $s_0$.

\begin{lemma}\label{lem:explicit_G_convex}
Let $\Omega\subset\R^d$ be a convex open set. Then, for all $\eps\in(0, 1]$, $r\in(0, \eps r_{in}(\Omega))$ and $s\in(0, r/2]$, 
  \begin{equation}
    |\{u\in \R^d: \dist(u, \partial\Omega)<s\}\setminus \G_{\eps, r}| 
    \leq C\Haus^{d-1}(\partial\Omega)\frac{s r}{\eps r_{in}(\Omega)} \, , 
  \end{equation}
  where $C$ depends only on the dimension.
\end{lemma}

\begin{proof}[Proof of Lemma~\ref{lem:explicit_G_convex}]
We divide the proof into three steps.

{\bf Step 1:} We define a set $G\subseteq\partial\Omega$.

We recall that $\Omega_t$ is defined in \eqref{eq:innerparallel}. We denote by $\reg(\partial\Omega_t)$ the set of points $x\in\partial\Omega_t$ for which the inner unit normal $\nu_t(x)$ exists. We consider the natural normal-map defined for $t\in [0, r_{in}(\Omega))$ by
\begin{equation}
  f_t\colon \reg(\partial \Omega_t)\times \R_\limplus \to \R^d \, , \  (x, s) \mapsto x- s\nu_t(x) \, .
\end{equation}
We observe that $f_{t}(\reg(\partial\Omega_{t}), s)\subseteq \reg(\partial\Omega_{t-s})$ for $0<s\leq t$ and, in particular, that $f_{t}(\reg(\partial\Omega_{t}), t)\subseteq \reg(\partial\Omega)$. We also note that for all $s\in [0, t]$ the inwards pointing normal to $\partial\Omega_{t-s}$ at $f_t(x, t-s)$ is equal to the normal at $x$, $\nu_t(x)$. It follows that the image of the map $f_t(x, \, \cdot \, )\colon [0, \infty)\to \R^d$ is a ray starting at $x$ and passing orthogonally through $\partial\Omega$ at the point $f_t(x, t)$. If $f_t(x, t)$ is $(\eps, r)$-good this ray forms the axis of symmetry for the cone $\Gamma_{\eps, r}(f_t(x, t))$. After these preparations, we now set
$$
G=f_{r/\eps}(\reg(\partial\Omega_{r/\eps}), r/\eps) \, .
$$

{\bf Step 2:} We show that for $\eps\in(0, 1)$ and $r\in (0, \eps r_{in}(\Omega))$ every $p\in G$ is $(\eps, r)$-good.

Note that we only need to check the $(\eps, r)$-condition in the inwards direction, since for any $y\in \reg(\partial\Omega)$ the boundary $\partial\Omega$ is contained in the half-space $\{u\in \R^d: (u-y)\cdot \nu(y)\geq 0\}$.

  The main idea behind the construction of $G$ is based on the observation that if a point $y\in \reg(\partial\Omega)$ fails to be $(\eps, r)$-good then it cannot be in the image of $f_t$ for suitably chosen~$t$, see Figure~\ref{fig:normalmap}. 

  Assume that $y\in \reg(\partial\Omega)$ fails to be $(\eps, r)$-good. If there is a point of $\reg(\partial\Omega_t)$ which is mapped to $y\in \reg(\partial\Omega)$ under the normal map $f_t$ it must be the point $y+t\nu(y)$. However, since $y$ is not $(\eps, r)$-good there is a point $y'\in \Omega^c$ such that $|y'-y|=r$ and $(y'-y)\cdot \nu(y)= \eps r$. By elementary trigonometry we find that if $t> \frac{r}{2\eps}$ then $|y+t\nu(y)-y'|<t$, and therefore $y+t\nu(t)$ does not belong to $\partial\Omega_t$ implying that $y\notin f_t(\reg(\partial\Omega_t), t)$. This proves that any $p\in G=f_{r/\eps}(\reg(\partial\Omega_{r/\eps}), r/\eps)$ is an $(\eps, r)$-good point of $\partial\Omega$.
  \begin{figure}[t!]
    \centering
    \input{normalmap}
    \caption{A $2$-dimensional cross-section of a neighborhood of $y$ illustrating the idea behind the construction of $G$. Here $\alpha=\arcsin(\eps)$ and $\alpha'=\pi/2-\alpha$.}
    \label{fig:normalmap}
  \end{figure}
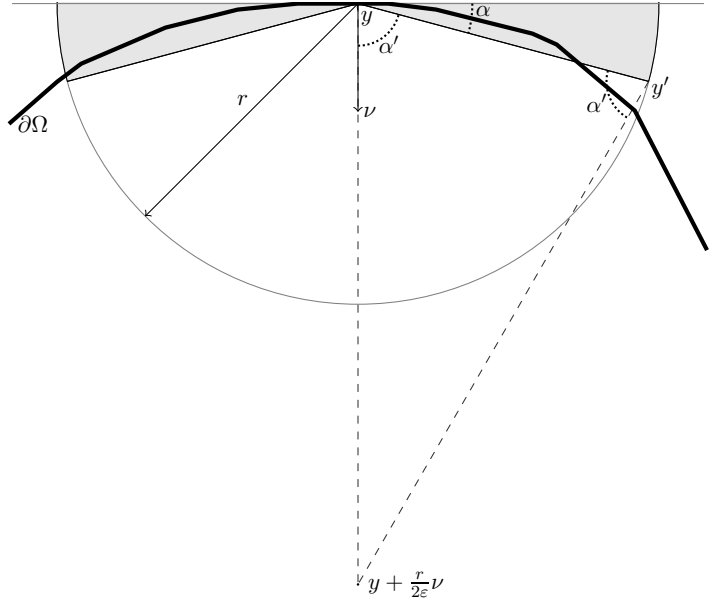

{\bf Step 3:} We now prove the inequality in the lemma.

We observe that for any fixed $t>0$ and all $s\geq 0$ the map $f_t(\, \cdot\, , s)$ is injective, and by convexity $\Haus^{d-1}(f_t(\reg(\partial\Omega_t), s))$ is an increasing functions of $s$. Note also that $\Haus^{d-1}(\reg(\partial\Omega_t))=\Haus^{d-1}(\partial\Omega_t)$ since $\Haus^{d-1}$-a.e.\ point of the boundary of a $d$-dimensional convex set is regular (see~\cite{Schneider_14}).

  Lemma~\ref{lem:boundary_region_convex} implies that
  \begin{align}
    |\{u\in \R^d: \dist(u, \partial\Omega)<s\}\setminus \G_{\eps, r}| 
    &\leq 
      2s \Haus^{d-1}(\partial\Omega)(1 + C s/r_{in}(\Omega))\\
    &\quad
     -|\{u\in \R^d: \dist(u, \partial\Omega)<s\}\cap \G_{\eps, r}| \, .
  \end{align}
Therefore using $s\leq r/2 \leq r/(2\eps)$ we see that the claimed inequality will follow from
  \begin{equation}
    |\{u\in \R^d:\dist(u, \partial\Omega)<s\}\cap\, \G_{\eps, r}|\geq 2s\Haus^{d-1}(\partial\Omega)\Bigl(1- \frac{C r}{\eps r_{in}(\Omega)}\Bigr) \, , \quad \forall s\leq r/2 \, .
  \end{equation}

  Since every $p\in G$ is $(\eps, r)$-good
  \begin{equation}
    f_{r/\eps}(\reg(\partial\Omega_{r/\eps}), r/\eps + s')\subset \G_{\eps, r} \, , 
    \quad \forall s'\in (-r/2, r/2) \, .
  \end{equation}
  Therefore, using again the co-area formula, \eqref{eq:Eikonal}, \eqref{eq:bound_perim_innerparallel} and the fact that $\Haus^{d-1}(f_t(\partial\Omega_t, s))$ is increasing in $s$, 
  \begin{align}
    |\{u\in \R^d: \dist(u, \partial\Omega)< s\}\cap \G_{\eps, r}|
    &\geq 
    \int_{-s}^{s}\Haus^{d-1}(f_{r/\eps}(\reg(\partial\Omega_{r/\eps}), r/\eps+s')\, ds'\\
    & \geq 
    2s \Haus^{d-1}(f_{r/\eps}(\reg(\partial\Omega_{r/\eps}), r/\eps-s)\\  
    &\geq 
    2s \Haus^{d-1}(\partial\Omega_{r/\eps})\\
    &\geq
    2s\Haus^{d-1}(\partial\Omega)\Bigl(1-\frac{r}{\eps r_{in}(\Omega)}\Bigr)^{d-1}\\
    &\geq
    2s \Haus^{d-1}(\partial\Omega)\Bigl(1-\frac{(d-1)r}{\eps r_{in}(\Omega)}\Bigr) \, .
  \end{align}
This completes the proof of Lemma~\ref{lem:explicit_G_convex}.
\end{proof}

\begin{remark}
  The points in the set $G$ in the previous proof are a lot better than $(\eps, r)$-good. The proof shows essentially that for any $p\in G$ the principal curvatures of $\partial\Omega$ are bounded from above by~$\sim \eps r^{-1}$. That this set is large for $r$ small enough follows from Aleksandrov's theorem on a.e.\ twice differentiability of convex functions.
\end{remark}

As explained at the beginning of this subsection, proving Theorem~\ref{thm:asymptotic_ineq_convex_alt} is now simply a matter of bounding all the relevant error terms in the derivation of the asymptotic expansion. 

\begin{proof}[Proof of Theorem~\ref{thm:asymptotic_ineq_convex_alt}]

We repeat the proof of Theorem~\ref{thm:MainTheorem_alt} but in~\eqref{eq:Contribution_badregion} and~\eqref{eq:Adding_bad_boundary}, where we used Lemma~\ref{lem:GoodSetIsBig}, we simply keep the term $|\Omega_b|$. In this way we find
  \begin{align}
   & h^{-d+1}\Bigl|\Tr(H_\Omega)_\limminus - L_d |\Omega|h^{-d}+ \frac{L_{d-1}}{4}\Haus^{d-1}(\partial\Omega)h^{-d+1}\Bigr| \\
    &\leq
  C \Haus^{d-1}(\partial\Omega)\biggl[\eps_0 \bigl[1+ \overline\Emink(r_{in}(\Omega))\bigr] + \frac{|\Omega_b|}{h\Haus^{d-1}(\partial\Omega)}
+ \overline\Emink(l_0)  + \bigl(\eps_0^{-1}\eps + \eps_0^{1/3}\bigr) \bigl[1+\overline\Emink(l_0)\bigr] \biggr] \, , 
  \end{align}
  where we again require $0<h<2r_{in}(\Omega)$, $r>0$, $\eps\in (0, 1/2]$ and $\eps_0\in(0, 4]$ to be chosen so that
  $$
  h/\eps_0 \leq \min\bigl\{r/2, r_{in}(\Omega)/2\bigr\} \, .
  $$

We now use the convexity of $\Omega$ to bound the terms which still depend on the geometry. By~\eqref{eq:inradius_bound} we have
$$
\overline\Emink(r_{in}(\Omega)) \leq C
\qquad\text{and}\qquad
\overline\Emink(l_0) \leq C \frac{l_0}{r_{in}(\Omega)} \, .
$$
Furthermore, if $r \leq \eps r_{in}(\Omega)$ and $l_0\leq r/2$, then Lemma~\ref{lem:explicit_G_convex} implies that  
    \begin{equation}
      |\Omega_b| \leq C \Haus^{d-1}(\partial\Omega) \frac{l_0 r}{\eps r_{in}(\Omega)} = C \Haus^{d-1}(\partial\Omega) \frac{h r}{\eps \eps_0 r_{in}(\Omega)} \, .
    \end{equation}
Therefore, the error term above is bounded by
$$
C \Haus^{d-1}(\partial\Omega)\biggl[ \frac{r}{\eps \eps_0 r_{in}(\Omega)} + \eps_0^{-1}\eps + \eps_0^{1/3} \biggr].
$$
(Here we have dropped a term $h/(\eps_0 r_{in}(\Omega))$ coming from the bound on $\overline\Emink(l_0)$, since $h\leq \eps_0 r\leq \eps_0 \eps r_{in}(\Omega)$, so this term is $\leq \eps$ and therefore also $\leq4\,\eps_0^{-1}\eps$.) The above bound is valid provided the parameters satisfy
$$
h\leq \eps_0\, r/2
\qquad\text{and}\qquad
r \leq \eps \, r_{in}(\Omega) \, .
$$

It remains to choose the parameters. We first assume that $s=h/r_{in}(\Omega)\leq 1$. Optimizing successively over $r$, $\eps$ and $\eps_0$ in that order and adjusting the constants we arrive at the choices
$$
r= (1/2) r_{in}(\Omega)\, s^{8/11} \, , 
\qquad
\eps = (1/2)\, s^{4/11} \, , 
\qquad
\eps_0 = 4\, s^{3/11} \, .
$$
Clearly all constraints are satisfied and the final error is
$$
C \Haus^{d-1}(\partial\Omega)\, s^{1/11} = C \Haus^{d-1}(\partial\Omega) (h/r_{in}(\Omega))^{1/11} \, .
$$
This is the claimed bound for $h\leq r_{in}(\Omega)$.

Finally, for any convex $\Omega\subset \R^d$ the first eigenvalue of $-\Delta_\Omega$ satisfies $\lambda_1(\Omega)\geq \frac{\pi^2}{4 r_{in}(\Omega)^2}$~\cite{Hersch, Protter_81}. Hence $\Tr(H_\Omega)_\limminus=0$ for all $h\geq (2/\pi) r_{in}(\Omega)$ and, in particular, for $h\geq r_{in}(\Omega)$. Combining this observation with the fact that $\frac{|\Omega|}{r_{in}(\Omega)} \leq \Haus^{d-1}(\partial\Omega)$ (see~\eqref{eq:inradius_bound}) the claimed bound holds also for any $h\geq r_{in}(\Omega)$, which completes the proof.
\end{proof}


\appendix

\section{Proof of Lemma~\ref{lem:localization}}
\label{appendixA}

What remains to conclude our analysis is to prove Lemma~\ref{lem:localization}. As mentioned earlier the proof follows the same strategy as the proof of Proposition~1.1 in~\cite{FrankGeisinger_11}.

\begin{proof}[Proof of Lemma~\ref{lem:localization}]
  Set
  \begin{equation}
    \gamma = \int_{\R^d}\phi_u( \phi_u \varphi H_\Omega \varphi \phi_u)_\limminus^0\phi_u l(u)^{-d}\,du \, .
  \end{equation}
  Clearly $\gamma \geq 0$ and by~\eqref{eq:phi_properties1} $\gamma \leq 1$. Since the range of $\gamma$ is a subset of $H^1_0(\Omega)$ the variational principle tells us that
  \begin{equation}
    \Tr(\varphi H_\Omega\varphi)_\limminus \geq - \Tr(\gamma \varphi H_\Omega \varphi) = \int_{\R^d}\Tr(\phi_u\varphi H_\Omega \varphi \phi_u)_\limminus l(u)^{-d}\,du\, .
  \end{equation}
  This completes the proof of one side of the inequality.

  To complete the proof we use the following version of the IMS-localization formula: for $f\in H_0^1(\Omega)$,
  \begin{equation}
    \frac{1}{2}( f, \phi_u^2 \varphi(-\Delta)\varphi f)+
    \frac{1}{2}(f, \varphi (-\Delta)(\phi_u^2\varphi f)) = (f, \phi_u\varphi(-\Delta) \varphi\phi_u f)-(\varphi f, \varphi f(\nabla \phi_u)^2)\, .
  \end{equation}
  By~\eqref{eq:phi_properties1} this yields that
  \begin{equation}\label{eq:IMS2}
    (f, \varphi(-\Delta)\varphi f) = \int_{\R^d}\bigl((f, \phi_u\varphi(-\Delta) \varphi\phi_u f)-(\varphi f, \varphi f(\nabla \phi_u)^2)\bigr) l(u)^{-d}\, du\, .
  \end{equation}

  Using the properties of $l$ and $\phi_u$ in Lemma~\ref{solovejspitzer} one can show, see the proof of~\cite[eq.~(68)]{SolovejSpitzer}, that
  \begin{equation}
    \int_{\R^d}(\nabla \phi_u)^2(x)l(u)^{-d}\,du \leq C \int_{\R^d}\phi_u^2(x)l(u)^{-d-2}\,du \, .
  \end{equation}
  When combined with~\eqref{eq:IMS2} we find that
  \begin{equation}\label{eq:local_trace_estimate}
    \Tr(\varphi H_\Omega \varphi)_\limminus \leq \int_{\dist(u, \Omega\,\cap \,\supp \varphi)\leq l(u)} \Tr(\phi_u\varphi(H_\Omega-C h^2l(u)^{-2})\varphi\phi_u)_\limminus l(u)^{-d}\,du\, .
  \end{equation}

  Let $0<\rho_u \leq 1$ be an additional parameter to be chosen later. By the variational principle
  \begin{equation}
  \begin{aligned}
     \Tr(&\phi_u\varphi(H_\Omega-C h^2l(u)^{-2})\varphi\phi_u)_\limminus\\
     &\leq 
     \Tr(\phi_u \varphi H_\Omega \varphi \phi_u)_\limminus + 
     \Tr(\phi_u \varphi (-\rho_u h^2 \Delta_\Omega-\rho_u-Ch^2l(u)^{-2})\varphi \phi_u)_\limminus\\
     &\leq
     \Tr(\phi_u \varphi H_\Omega \varphi \phi_u)_\limminus 
     + 
     L_d (\rho_u+Ch^2l(u)^{-2})^{1+d/2}\rho_u^{-d/2}h^{-d}\int_{\Omega}\phi_u^2(x)\varphi(x)^2 \,dx\, ,
   \end{aligned} 
   \end{equation}
   where we in the last step used Lemma~\ref{lem:BerezinLiebLiYau}.

   Setting $\rho_u = h^2 l(u)^{-2}/M^2$, which by assumption is bounded by $1$, we conclude that 
   \begin{align}\label{eq:localization_proof_berezin}
     \Tr(\phi_u&\varphi(H_\Omega-C h^2l(u)^{-2})\varphi\phi_u)_\limminus\\
     &
     \leq
     \Tr(\phi_u \varphi H_\Omega \varphi \phi_u)_\limminus
     +  L_dM^{-2}(1+C M^2)^{1+d/2}h^{-d+2}l(u)^{-2}\int_{\Omega}\phi_u^2(x)\varphi(x)^2\,dx \, .
   \end{align}

   Since $\|\phi_u\|_{L^\infty}\leq C$ and $|{\supp \phi_u}| \leq Cl(u)^d$ it holds that
   \begin{align}\label{eq:Linfty_estimate}
     \int_{\dist(u, \Omega \, \cap \, \supp \varphi)\leq l(u)}&\int_\Omega \phi_u(x)^2 \varphi(x)^2 l(u)^{-d-2}\,dx\,du\\
     &\leq
     \|\varphi\|_{L^\infty(\Omega)}^2\int_{\dist(u, \Omega \, \cap \, \supp \varphi)\leq l(u)}\int_\Omega \phi_u(x)^2 l(u)^{-d-2}\,dx\,du\\
     &\leq
     C\|\varphi\|_{L^\infty(\Omega)}^2\int_{\dist(u, \Omega \, \cap \, \supp \varphi)\leq l(u)}l(u)^{-2}\,du \, .
   \end{align}

   Combining~\eqref{eq:local_trace_estimate},~\eqref{eq:localization_proof_berezin} and~\eqref{eq:Linfty_estimate} completes the proof of the lemma.
\end{proof}

We now move on to proving that the inequality of Proposition~1.1 in~\cite{FrankGeisinger_11} can be extended to all~$h>0$. We also show that the same construction allows us to prove the analogous statement for the length scale used in the proof of Theorem~\ref{thm:MainTheorem}.

We begin with a function $l$ as in Lemma~\ref{solovejspitzer} and any constant $S>0$. Assuming that $h\geq S \max_{\dist(u, \Omega\, \cap\, \supp \varphi)\leq l(u)}l(u)$ then by Lemma~\ref{lem:BerezinLiebLiYau} and~\eqref{eq:phi_properties1}
\begin{equation}\label{eq:h_large}
\begin{aligned}
  \Bigl|\Tr(\varphi &H_\Omega \varphi)_\limminus-\int_{\R^d}\Tr(\phi_u\varphi H_\Omega \varphi\phi_u)_\limminus l(u)^{-d}\, du\Bigr| \\
  &\leq 
  h^{-d}L_d\int_{\Omega} \varphi^2(x)\,dx+ h^{-d}L_d \int_{\R^d}\int_\Omega\varphi^2(x)\phi_u^2(x)l(u)^{-d}\,dx\,du\\
  &=
  h^{-d}2L_d\int_{\dist(u, \Omega\, \cap\, \supp \varphi)\leq l(u)}\int_\Omega\varphi^2(x)\phi_u^2(x)l(u)^{-d}\,dx\,du\\
  &\leq
  h^{-d}C \|\varphi\|^2_{L^\infty(\Omega)} \int_{\dist(u, \Omega \, \cap \supp \varphi)\leq l(u)} \,du\\
  &\leq
  h^{-d+2}C \|\varphi\|^2_{L^\infty(\Omega)} S^{-2} \int_{\dist(u, \Omega\, \cap \,\supp \varphi)\leq l(u)}l(u)^{-2}\,du \, .
\end{aligned}
\end{equation}
Here we used that $\int_\Omega \varphi^2(x)\phi_u(x)^2\,dx \leq \|\varphi\|^2_{L^\infty}C l(u)^{d}$ to obtain an estimate which matches that of Lemma~\ref{lem:localization}.

Assume now that we are given a length scale $l$ depending on a parameter $l_0$, which itself depends on $h$ in such a way that there are constants $\delta, \mu>0$ such that for $h \leq \delta$ one has $l_0 \geq \mu h$.

We first consider the length scale used in~\cite{FrankGeisinger_11}:
\begin{equation}
  l(u) = \frac{1}{2}\bigl(1+ (\dist(u, \Omega^c)+l_0^2)^{-1/2}\bigr)^{-1}, \quad \mbox{with }0<l_0\leq 1\, .
\end{equation}
We have that 
\begin{align}
  \min_{\dist(u, \Omega)\leq l(u)} l(u) &= \frac{l_0}{2+2l_0}\, ,\\
  \max_{\dist(u, \Omega)\leq l(u)} l(u) &\leq 1/2\, .
\end{align}
If $h\leq \delta$ and we set $M=\frac{2+\mu\delta}{\mu}$ then
\begin{equation}
  M\min_{\dist(u, \Omega)\leq l(u)} l(u) = \frac{2+2\mu\delta}{\mu}\frac{l_0}{2+2l_0} \geq \frac{2+2\mu\delta}{\mu}\frac{\mu h}{2+2\mu h} \geq h \, .
\end{equation}
Therefore, we can in the regime $h\leq \delta$ apply Lemma~\ref{lem:localization} with $M$ as above.
On the other hand, if $h> \delta$ then with $S=2\delta$ we have    
\begin{equation}
  S \max_{\dist(u, \Omega)\leq l(u)} l(u) \leq 2\delta /2 < h \, .
\end{equation}
Thus if $h>\delta$ we can apply~\eqref{eq:h_large} with $S=2\delta$.
In conclusion, with the choices of $l$ and $l_0$ made in~\cite{FrankGeisinger_11} the claimed inequality is valid for all~$h>0$.

Similarly, for the length scale~\eqref{eq:lengthscale_mainproof} used in the proof of Theorem~\ref{thm:MainTheorem}
we have
\begin{align}
  \min_{\dist(u, \Omega)\leq l(u)} l(u) &= l_0\, ,\\
  \max_{\dist(u, \Omega)\leq l(u)} l(u) &\leq r_{in}(\Omega)/2\, .
\end{align}
Setting $M=1/\mu$ and $S=2\delta/r_{in}(\Omega)$ we find
\begin{align}
  M\min_{\dist(u, \Omega)\leq l(u)} l(u) &= l_0/\mu \geq h\, ,\qquad  \mbox{for } h \leq \delta\, ,\\
  S \max_{\dist(u, \Omega)\leq l(u)} l(u) &\leq \delta < h\, , \hspace{35pt}  \mbox{for } h > \delta\, ,
\end{align}
and we can conclude in the same manner as above.


\bibliographystyle{amsplain}

\def\myarXiv#1#2{\href{http://arxiv.org/abs/#1}{\texttt{arXiv:#1\, [#2]}}}

\end{document}

%% file: cones
\begin{tikzpicture}[scale=1]

\clip (-7,-2.8) rectangle (7,2.8);

\def \angle {15};
\def \lll {5.5};
\def \wid {0.13};

\draw[gray] (0,0) circle [radius=\lll];
\draw[gray] (0,0) circle [radius={\lll/2}];

\draw[fill=gray!20] (0,0) -- ({-\lll*sin(\angle)/2},{-\lll*cos(\angle)/2}) arc ({270-\angle}:{270+\angle}:{\lll/2})--cycle;

\draw[fill=gray!20] (0,0) -- ({\lll*sin(\angle)/2},{\lll*cos(\angle)/2}) arc ({90-\angle}:{90+\angle}:{\lll/2})--cycle;

\draw[->] (0, 0)--(0, -2.2);
\draw[gray] (-5.75,0)--(5.75,0);

\draw[thick] ({-(\lll*1.1)*cos(\angle)},{-(\lll*1.1)*sin(\angle)})--(0,0)--({(\lll*1.1)*cos(\angle)},{-(\lll*1.1)*sin(\angle)});
\fill[pattern=vertical lines, pattern color=black] (0,0) -- ({-\lll*cos(\angle)*1.1},{-\lll*sin(\angle)*1.1})--({-\lll*cos(\angle)*1.1},{-\lll*sin(\angle)*1.1-\wid})--(0,-\wid)--({(\lll*1.1)*cos(\angle)},{-(\lll*1.1)*sin(\angle)-\wid})--({(\lll*1.1)*cos(\angle)},{-(\lll*1.1)*sin(\angle)})--cycle;

\draw[thick, blue] ({-(\lll*1.1)*cos(\angle)},{(\lll*1.1)*sin(\angle)})--(0,0)--({(\lll*1.1)*cos(\angle)},{(\lll*1.1)*sin(\angle)});
\fill[pattern=vertical lines, pattern color=blue] (0,0) -- ({-\lll*cos(\angle)*1.1},{\lll*sin(\angle)*1.1})--({-\lll*cos(\angle)*1.1},{\lll*sin(\angle)*1.1-\wid})--(0,-\wid)--({(\lll*1.1)*cos(\angle)},{(\lll*1.1)*sin(\angle)-\wid})--({(\lll*1.1)*cos(\angle)},{(\lll*1.1)*sin(\angle)})--cycle;

\draw[black!50,thick] ({-\lll*sin(\angle)/2},{-\lll*cos(\angle)/2})--({\lll*sin(\angle)/2},{\lll*cos(\angle)/2});
\draw[black!50,thick] ({-\lll*sin(\angle)/2},{\lll*cos(\angle)/2})--({\lll*sin(\angle)/2},{-\lll*cos(\angle)/2});

\draw [thick,densely dotted, samples=50,domain=0:\angle] plot ({1.5*cos(\x)}, {1.5*sin(\x)});
\draw [thick,densely dotted, samples=50,domain=270:270+\angle] plot ({1.5*cos(\x)}, {1.5*sin(\x)});

\draw [ultra thick,samples=10,domain=-6:-2] plot ({\x}, {sin(\x*360)/\x+ \x*\x/20-1/2});
\draw [ultra thick](-2,{sin(2*360)/2+ 1/5-1/2})--({-1*cos(\angle)}, {sin(\angle)});
\draw [ultra thick]({-1*cos(\angle)}, {sin(\angle)})--(-0.8,-0.15);
\draw [ultra thick](-0.8,-0.15)--(-0.4,0)--(0.2,0)--(0.5, -0.1)--(1, 0)--(2.1,-1/2)--(3.5,1/3)--(5, 1)--(5.8, 1.9);

\draw[thick,red] ({-\lll*cos(12)*1.1},{-\lll*sin(11)*1.1})--({\lll*cos(11)*1.1},{\lll*sin(11)*1.1});

\fill[pattern=vertical lines, pattern color=red] ({-\lll*cos(12)*1.1},{-\lll*sin(11)*1.1})--({\lll*cos(11)*1.1},{\lll*sin(11)*1.1})--({\lll*cos(11)*1.1},{\lll*sin(11)*1.1-\wid})--({-\lll*cos(12)*1.1},{-\lll*sin(11)*1.1-\wid})--cycle;

\node at (0.2,1/4) {{$p$}};
\node at (-6, 0.9) {{$\partial\Omega$}};
\node at (-4.5, -0.65) {{$\partial L^*$}};
\node at (0.2,-2.2) {{$\nu$}};
\node at (1.7,0.2) {{$\alpha$}};
\node at (0.2,-1.7) {{$\alpha$}};
\node at (0,2) {{$\Gamma_{r,\eps}$}};
\node at (-3.2,-1.15) {{$\partial\I$}};
\node at (-3.2,1.1) {{$\partial\U$}};
\node at (-0.3,-2.5) {{$\cdot u$}};

\draw[<->] (-1,-0.28)--(-0.417,-2.48)--(-0.81,-0.15);

\draw[<->] ({\lll*cos(30)},{-\lll*sin(30)})--({\lll*cos(30)/2},{-\lll*sin(30)/2});
\draw[<->] (0,0)--({\lll*cos(30)/2},{-\lll*sin(30)/2});

\node at ({\lll*cos(41)/3.3},{-\lll*sin(41)/3.3}) {{$r/2$}};
\node at ({3*\lll*cos(34.5)/4},{-3*\lll*sin(34.5)/4}) {{$r/2$}};

\node at (-0.9,-1.3) {{$l'$}};
\node at (-0.5,-1) {{$l$}};

\end{tikzpicture}

%% file: normalmap
\def \tmpscale {0.8}

\begin{tikzpicture}[scale=\tmpscale]

\def \angle {15};
\def \lll {5};
\def \wid {0.13};

\clip (-7,0.025) rectangle (7,{-\lll/(2*sin(\angle))-0.3});

\draw[gray] (0,0) circle [radius=\lll];

\draw[fill=gray!20] (0,0) -- ({-\lll*cos(\angle)},{\lll*sin(\angle)}) arc ({180-\angle}:{180+\angle}:{\lll})--cycle;
\draw[fill=gray!20] (0,0) -- ({\lll*cos(-\angle)},{-\lll*sin(\angle)}) arc ({-\angle}:{+\angle}:{\lll})--cycle;

\draw[dashed, black!80] (0, 0)--(0, {-\lll/(2*sin(\angle))});
\draw[->] (0, 0)--(0, -1.8);
\node at (0.8, {-\lll/(2*sin(\angle))}) {\scalebox{\tmpscale}{$y+\frac{r}{2\eps}\nu$}};
\fill (0, {-\lll/(2*sin(\angle))}) circle [radius=0.021];

\draw[gray] (-5.75,0)--(5.75,0);
\def \tmp {5};
\draw[dashed, black!80] ({\tmp*cos(\angle)}, {-\tmp*sin(\angle)})--({\tmp*cos(\angle)+\lll*cos(270-2*\angle)/(2*sin(\angle))}, {-\tmp*sin(\angle)+\lll*sin(270-2*\angle)/(2*sin(\angle))});

\draw [thick,densely dotted, samples=50,domain=180-\angle:270-2*\angle] plot  ({\tmp*cos(\angle)+0.7*cos(\x)}, {-\tmp*sin(\angle)+0.7*sin(\x)});
\node at ({\tmp*cos(\angle)+0.928*cos(225-3/2*\angle)}, {-\tmp*sin(\angle)+0.928*sin(225-3/2*\angle)}) {\scalebox{\tmpscale}{$\alpha'$}};

\draw [thick,densely dotted, samples=50,domain=270:360-\angle] plot  ({0.7*cos(\x)}, {0.7*sin(\x)});
\node at ({0.85*cos(270+45-\angle/2)}, {0.85*sin(270+45-\angle/2)}) {\scalebox{\tmpscale}{$\alpha'$}};

\draw ({-(\lll*1)*cos(\angle)},{-(\lll*1)*sin(\angle)})--(0,0)--({(\lll*1)*cos(\angle)},{-(\lll*1)*sin(\angle)});
\draw ({-(\lll*1)*cos(\angle)},{(\lll*1)*sin(\angle)})--(0,0)--({(\lll*1)*cos(\angle)},{(\lll*1)*sin(\angle)});

\draw [thick,densely dotted, samples=50,domain=0:-\angle] plot ({1.9*cos(\x)}, {1.9*sin(\x)});
\node at (2.08,-0.13) {\scalebox{\tmpscale}{$\alpha$}};

\draw [ultra thick]
(-5.8,-2)--(-5,-1.3)--(-4.6,-1)--(-3.2,-0.4)--(-2, -0.1)--(-1,0)--(0.5,0)--(1.3,-0.1)--(2.9,-0.5)--(3.3,-0.68)--(4.6,-1.78)--(5.8, -4.1);

\node at (0.15,-1/4) {\scalebox{\tmpscale}{$y$}};
\node at ({\tmp*cos(\angle)+0.2}, {-\tmp*sin(\angle)-0.1}) {\scalebox{\tmpscale}{$y'$}};
\node at (-5.4, -2) {\scalebox{\tmpscale}{$\partial\Omega$}};
\node at (0.2,-1.8) {\scalebox{\tmpscale}{$\nu$}};

\draw [<->] (0,0)--({\lll*cos(225)},{\lll*sin(225)});
\node at ({\lll*cos(220)/2},{\lll*sin(220)/2}) {\scalebox{\tmpscale}{$r$}};

\end{tikzpicture}

%% file: Main.bbl
\begin{thebibliography}{10}

\bibitem{AizenmanLieb_78}
M.~Aizenman and E.~H. Lieb, \emph{On semiclassical bounds for eigenvalues of
  {S}chr\"odinger operators}, Phys. Lett. A \textbf{66} (1978), no.~6, 
  427--429.


\bibitem{AmbrosioEtAl_08}
L.~Ambrosio, A.~Colesanti, and E.~Villa, 
\emph{Outer {M}inkowski content for some classes of closed sets}, Math. Ann. \textbf{342} (2008), no.~4, 727--748.

\bibitem{BanuelosEtAl_09}
R.~Ba\~nuelos, T.~Kulczycki, and B.~Siudeja, \emph{On the trace of symmetric
  stable processes on {L}ipschitz domains}, J. Funct. Anal. \textbf{257}
  (2009), no.~10, 3329--3352.


\bibitem{Berezin}
F.~A. Berezin, \emph{Covariant and contravariant symbols of operators}, 
   Izv. Akad. Nauk SSSR Ser. Mat. \textbf{36} (1972), 1134--1167.


\bibitem{vdBerg_84}
M.~{van den Berg}, \emph{A uniform bound on trace{$\, (e^{t\Delta })$} for convex 
  regions in {${\R}^{n}$} with smooth boundaries}, 
    Comm. Math. Phys. \textbf{92} (1984), no.~4, 525--530.

\bibitem{vdBerg_87}
M.~{van den Berg}, \emph{On the asymptotics of the heat equation and bounds on traces associated with the Dirichlet Laplacian}, J. Funct. Anal. \textbf{71} (1987), no.~2, 279--293.



\bibitem{BronsteinIvrii_03}
M. Bronstein and V. Ivrii, \textit{Sharp spectral asymptotics for operators with irregular coefficients. I. Pushing the limits}. Comm. Partial Differential Equations \textbf{28} (2003), no. 1--2, 83--102.



\bibitem{Brown_93}
R.~M. Brown, \emph{The trace of the heat kernel in {L}ipschitz domains}, Trans.
  Amer. Math. Soc. \textbf{339} (1993), no.~2, 889--900.


\bibitem{Federer_book}
H.~Federer, \emph{Geometric measure theory}, Die Grundlehren der mathematischen Wissenschaften, vol.~153, Springer-Verlag, New York, 1969.

\bibitem{FrankGeisinger_11}
R.~L. Frank and L.~Geisinger, \emph{Two-term spectral asymptotics for the
  {D}irichlet {L}aplacian on a bounded domain}, Mathematical results in quantum
  physics, World Sci. Publ., Hackensack, NJ, 2011, pp.~138--147.

\bibitem{FrankGeisinger_12}
R.~L. Frank and L.~Geisinger, \emph{Semi-classical analysis of the {L}aplace operator with {R}obin
  boundary conditions}, Bull. Math. Sci. \textbf{2} (2012), no.~2, 281--319.

\bibitem{FrankGeisinger_16}
R.~L. Frank and L.~Geisinger, \emph{Refined semiclassical asymptotics for fractional powers of the
              {L}aplace operator}, J. Reine Angew. Math. \textbf{712} (2016), 1--37.


\bibitem{GeisingerLaptevWeidl_11}
L.~{Geisinger}, A.~{Laptev}, and T.~{Weidl}, \emph{{Geometrical Versions of
  improved {B}erezin-{L}i-{Y}au Inequalities}}, Journal of Spectral Theory
  \textbf{1} (2011), 87--109.

\bibitem{GeisingerWeidl_10}
L.~Geisinger and T.~Weidl, \emph{Universal bounds for traces of the {D}irichlet
  {L}aplace operator}, J. Lond. Math. Soc. \textbf{82} (2010), no.~2, 
  395--419.

\bibitem{HarrellProvenzanoStubbe_18}
E. M. Harrell II, L. Provenzano, and J. Stubbe, \emph{Complementary asymptotically sharp estimates for eigenvalue means of Laplacians}, Int. Math. Res. Not. IMRN (2019), rnz085.


\bibitem{HarrellStubbe_16}
E. M. Harrell II and J. Stubbe, \emph{Two-term asymptotically sharp estimates for eigenvalue means of the Laplacian}, Journal of Spectral Theory \textbf{8} (2018), 1529--1550.


\bibitem{Hersch}
J.~Hersch, \emph{Sur la fr\'equence fondamentale d'une membrane vibrante:
              \'evaluations par d\'efaut et principe de maximum}, 
              Z. Angew. Math. Phys. \textbf{11} (1960), 387--413.

\bibitem{Ivrii_80}
V. Ivrii, \emph{The second term of the spectral asymptotics for a
  {L}aplace-{B}eltrami operator on manifolds with boundary}, Funktsional. Anal.
  i Prilozhen. \textbf{14} (1980), no.~2, 25--34.

\bibitem{Ivrii_00}
V. Ivrii, \emph{Sharp spectral asymptotics for operators with irregular coefficients}, Int. Math. Res. Not. \textbf{2000} (2000), no.~22, 1155--1166.

\bibitem{Ivrii_03}
V. Ivrii, \textit{Sharp spectral asymptotics for operators with irregular coefficients. II. Domains with boundaries and degenerations}. Comm. Partial Differential Equations \textbf{28} (2003), no. 1--2, 103--128.


\bibitem{KovarikVugalterWeidl_09}
H.~Kova{\v{r}}{\'{\i}}k, S.~Vugalter, and T.~Weidl, \emph{Two-dimensional
  {B}erezin--{L}i--{Y}au inequalities with a correction term}, Comm. Math. Phys.
  \textbf{287} (2009), no.~3, 959--981.

\bibitem{Larson_JFA_16}
S.~Larson, \emph{{A bound for the perimeter of inner parallel bodies}}, J.
  Funct. Anal. \textbf{271} (2016), no.~3, 610--619.

\bibitem{LarsonPAMS}
S.~Larson, \emph{{On the remainder term of the Berezin inequality on a convex domain}}, Proc. Amer. Math. Soc., \textbf{145} (2017), no.~5, 2167--2181.

\bibitem{Larson_JST}
S.~Larson, \emph{{Asymptotic shape optimization for Riesz means of the Dirichlet
  Laplacian over convex domains}}, Journal of Spectral Theory (to appear).


\bibitem{LiYau_83}
P. Li and S.~T. Yau, \emph{On the {S}chr\"odinger equation and the eigenvalue problem}, 
      Comm. Math. Phys. \textbf{88} (1983), no.~3, 309--318.


\bibitem{Melas_03}
A.~D. Melas, \emph{A lower bound for sums of eigenvalues of the {L}aplacian}, 
    Proc. Amer. Math. Soc. \textbf{131} (2003), no.~2, 631--636.


\bibitem{Protter_81}
M. H. Protter, \emph{A lower bound for the fundamental frequency of a convex
              region}, Proc. Amer. Math. Soc. \textbf{81} (1981), no.~1, 65--70.

\bibitem{Rozenblum_72}
G. V. Rozenblum, \emph{On the eigenvalues of the first boundary value problem in unbounded domains}. Math. USSR Sb. \textbf{18} (1972), no. 2, 235--248.


\bibitem{Schneider_14}
R.~Schneider, \emph{{C}onvex bodies: the {B}runn--{M}inkowski theory}, {S}econd
  {E}xpanded ed., Encyclopedia of Mathematics and its Applications, vol. 151, 
  Cambridge Univ. Press, Cambridge, 2014.

\bibitem{Seeley_78}
R. Seeley, \emph{A sharp asymptotic remainder estimate for the eigenvalues of the Laplacian in a domain of $\R^3$}. Adv. in Math. \textbf{29} (1978), no. 2, 244--269.

\bibitem{Seeley_80}
R. Seeley, \emph{An estimate near the boundary for the spectral function of the Laplace operator}. Amer. J. Math. \textbf{102} (1980), no. 5, 869--902.

\bibitem{SolovejSpitzer}
J.~P.~Solovej and W.~L.~Spitzer, 
\emph{A new coherent states approach to semiclassics which gives
              Scott's correction}, Comm. Math. Phys. \textbf{241} (2003), no.~2-3, 383--420.

\bibitem{Vassiliev_84}
D. Vassiliev,
\emph{Two-term asymptotics of the spectrum of a boundary value problem under
an interior reflection of general form}. Functional Anal. Appl. \textbf{18} (1984), 267--277.

\bibitem{Vassiliev_86}
D. Vassiliev,
\emph{Two-term asymptotics of the spectrum of a boundary value problem in the case of a piecewise smooth boundary}. Soviet Math. Dokl. \textbf{33} (1986), no. 1, 227--230.

\bibitem{Weidl_08}
T.~Weidl, \emph{Improved {B}erezin--{L}i--{Y}au inequalities with a remainder
  term}, Spectral theory of differential operators, Amer. Math. Soc. Transl.
  Ser.~2, vol.~225, pp.~253--263, 2008.

\bibitem{Weyl_12}
H.~Weyl, \emph{Das asymptotische {V}erteilungsgesetz der {E}igenwerte linearer
  partieller {D}ifferentialgleichungen \textnormal{(}mit einer {A}nwendung auf
  die {T}heorie der {H}ohlraumstrahlung\textnormal{)}}, Math. Ann. \textbf{71}
  (1912), no.~4, 441--479.

\bibitem{Weyl_13}
H.~Weyl, \emph{\"Uber die {R}andwertaufgabe der {S}trahlungstheorie und
  asymptotische {S}pektralgesetze}, J. Reine Angew. Math. \textbf{143} (1913), 
  177--202.

\end{thebibliography}
